\documentclass{article}[12pt]
\usepackage{url, graphicx, xcolor, hyperref, geometry}
\usepackage{latexsym,amsbsy}
\usepackage{lipsum}
\usepackage{amsfonts}
\usepackage{graphicx}
\usepackage{epstopdf}
\usepackage{algorithmic}
\usepackage{bm}
\usepackage{tikz}
\usetikzlibrary{decorations.pathreplacing,calligraphy}
\usetikzlibrary{matrix}
\usetikzlibrary{shapes,patterns,arrows,snakes,decorations.shapes}
\usetikzlibrary{positioning,fit,calc}
\usetikzlibrary{plotmarks}
\usepackage{amssymb,amsmath,amsthm}
\usepackage{smartdiagram}
\usepackage{mathabx}
\usepackage{float}
\usepackage{enumerate}
\usepackage{diagbox}
\usepackage{todonotes}
\usepackage{multirow}
\newtheorem{theorem}{Theorem}[section]
\newtheorem{lemma}{Lemma}[section]

\newtheorem{remark}{Remark}[section]

\usepackage[font=small,labelfont=bf, labelsep=space]{caption}
\hypersetup{
    colorlinks,
    linkcolor={blue},
    citecolor={blue},
    urlcolor={blue}
}

\begin{document}
 
\title{Block Preconditioners for the Marker-and-Cell Discretization of the Stokes--Darcy Equations}

\author{Chen Greif\thanks{Department of Computer Science, The University of British Columbia, Vancouver, BC, V6T 1Z4, Canada. The work of the first author was supported in part by a Discovery Grant of the Natural Sciences and Engineering Research Council of Canada. \tt{greif@cs.ubc.ca},  \tt{yunhui.he@ubc.ca}. } 
\and Yunhui He\footnotemark[1]}

\maketitle

\begin{abstract}
We consider the problem of iteratively solving large and sparse double saddle-point  systems arising from the stationary Stokes--Darcy equations in two dimensions, discretized by the Marker-and-Cell (MAC) finite difference method.  We analyze the eigenvalue distribution of a few ideal block preconditioners. We then derive practical preconditioners that are based on approximations of Schur complements that arise in a block decomposition of the double saddle-point  matrix. We show that including the interface conditions in the preconditioners is key in the pursuit of scalability. Numerical results show good convergence behavior of our preconditioned GMRES solver and demonstrate robustness of the proposed preconditioner with respect to the physical parameters of the problem. 
\end{abstract}

\vskip0.3cm {\bf Keywords.}
 Stokes--Darcy equations, Marker-and-Cell, double saddle-point systems,  iterative solution,   preconditioning, eigenvalues

\vskip0.3cm {\bf AMS.}
65F08, 65F10, 65N06


\section{Introduction}\label{sec:intro}

The numerical solution of coupled fluid problems has attracted a considerable attention of researchers and practitioners in the past few decades, in large part due to the importance of these problems and the computational challenges that they pose. The Stokes--Darcy model is an example of such a problem, and is the topic of this paper. The equations describe the flow of fluid across two subdomains: in one subdomain the fluid flows freely, and in the other it flows through a porous medium. The interface between the subdomains couples the two flow regimes and plays a central physical, mathematical, and computational role. It poses a challenge because the flow behaves significantly differently in terms of scale and other properties in each of the subdomains, and an abrupt change of scale may occur at the interface.
There are several relevant applications of interest here: flow of water through sand and rock, flow of blood through arterial vessels, problems in hydrology, environment and climate science, and other applications; see, e.g., the comprehensive survey~\cite{discacciati2009navier}. 

As far as the numerical solution of the equations is concerned, different types of discretizations have been investigated; for example, finite element methods \cite{karper2009unified,zhang2009low,chen2016weak,marquez2015strong,bernardi2008mortar}, finite difference/volume methods \cite{rui2020mac,shiue2018convergence,luo2017uzawa}, and other methods \cite{wang2019divergence,fu2018strongly}.  
In addition to methods that solve the problem for the entire domain at once, there are also domain decomposition methods or iteration-by-subdomain methods, which solve separately the Stokes and the Darcy problems in an iterative fashion \cite{discacciati2007robin}.
See also \cite{riviere2005locally,layton2002coupling,cao2010coupled,caiazzo2014classical,
babuvska2010residual,mu2007two,hessari2015pseudospectral,keyes2013multiphysics} and the references therein.

The Marker-and-Cell (MAC) scheme belongs to the class of finite difference methods, and is our  focus in this work.  MAC was proposed  in \cite{harlow1965numerical}   for the Stokes and Navier--Stokes equations. To achieve numerical stability, the scheme uses  staggered grids in which the velocity and pressure are discretized at different locations of a grid cell.   MAC  has been used extensively for fluid flow problems, and significant effort has been devoted to the study of this scheme for the coupled Navier--Stokes and Darcy flows \cite{lai2019simple},   Stokes--Darcy--Brinkman equations \cite{sun2019stability},  the compressible Stokes equations \cite{eymard2010convergence}, and other multiphysics applications \cite{liu2001energy,duretz2011discretization}.  A review of the Marker-and-Cell method  can be found in \cite{mckee2008mac}. 

As shown in \cite{nicolaides1992analysis,mckee2004recent, rui2020mac} and several other references, the MAC scheme has a few advantages. It is well-tested and well-understood for standard fluid flow problems, and it allows for a relatively simple  implementation. 
For the Stokes problem, it has been shown that the MAC method can be derived 
directly from  a finite element method \cite{HW1998}.
For the Navier--Stokes problem, the MAC method can be interpreted as a mixed finite element
 method of the velocity-vorticity
variational formulation \cite{GL1996}.
 Recent papers prove numerical stability and convergence of the Stokes--Darcy problems \cite{shiue2018convergence, sun2019stability}. In this paper we use the discretization introduced in \cite{shiue2018convergence}.

Preconditioners for GMRES  for  the   mixed Stokes--Darcy model discretized by mixed finite element method have been proposed in \cite{cai2009preconditioning}. In \cite{chidyagwai2016constraint} an indefinite constraint preconditioner was studied. In \cite{beik2022preconditioning} an augmented Lagrangian approach is used, and a field-of-values analysis is performed.
For multigrid solvers, the main challenge is in designing effective smoothers for the coupled discrete systems.  In \cite{luo2017uzawa}, the authors develop an Uzawa smoother for the Stokes--Darcy problem discretized by finite volumes on staggered grids.
The recent paper~\cite{mardal2022robust} provides an interesting description of some challenges that arise with various formuations of the problem. The authors show that 
standard preconditioning approaches based on natural norms are not parameter-robust, and they propose preconditioners that 
utilize non-standard and non-local operators, which are based on fractional derivatives. 
 For additional useful references on solution approaches for solving the problem, see \cite{schmalfuss2021partitioned,beik2022preconditioning}.

In this work, we focus on preconditioning for the stationary Stokes--Darcy problem discretized by the MAC scheme. We propose block-structured preconditioners, perform a spectral analysis of the preconditioned operators, and show that  they  are suitable for preconditioned GMRES. Taking advantage of the sparsity structure of the matrix and using effectively the coupling equations, we develop inexact approximations of the Schur complements and show that the iterative scheme performs robustly.

In Section \ref{sec:equations} we review  the continuous Stokes--Darcy equations and in Section \ref{sec:discretization} we describe the MAC scheme for discretizing  them.  We develop  block preconditioners and their inexact versions in Section \ref{sec:preconditioner}.  In Section \ref{sec:numerical} numerical results are presented. Finally, we draw some conclusions in Section \ref{sec:conclusion}. 
  
%

\section{Governing equations}\label{sec:equations}

We consider the coupled Stokes--Darcy problem in a two-dimensional domain comprised of two non-overlapping subdomains, $\Omega = \Omega_d \bigcup\Omega_s$; see Figure~\ref{fig:domain}. In the bounded domain $\Omega_s$ we have a free fluid flow, and in $\Omega_d$ the flow is in a porous region. The flows are coupled across the interface $\Gamma$. 

 \begin{figure}
 \begin{center}
\includegraphics[width=0.45\textwidth]{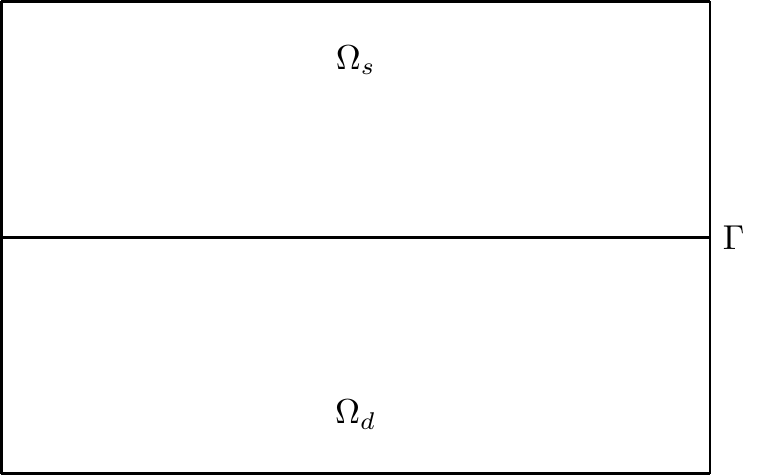}
\caption{Two-dimensional domain for the Stokes--Darcy problem. The interface is marked by $\Gamma$. \label{fig:domain}} 
\end{center}
\end{figure}

The Darcy equations in two dimensions  for porous medium flow are given by
\begin{subequations}
\label{eq:Darcy}
\begin{align}
  {K}^{-1}\bm{u}^d+\nabla p^d&=0 \quad  {\rm in}\,\,\Omega_d,  \label{eq: Darcy-form1}\\
  \nabla\cdot \bm{u}^d&= f^d \quad  {\rm in}\,\,  \Omega_d, \label{eq: Darcy-form2}
\end{align}
\end{subequations}
where $\bm{u}^d=( u^d, v^d )$ is the velocity and  $p^d$ is the fluid pressure inside the porous medium. ${K}$ is the hydraulic (or permeability) tensor, representing the properties of the porous medium and the fluid. Throughout this paper  we will assume ${K}=\kappa {I}$, where $\kappa>0$ and $I$ is the identity matrix. This amounts to treating the porous medium  as homogeneous and isotropic, and we call $\kappa$ the permeability constant. 

Denoting $\phi=p^d$ we rewrite \eqref{eq: Darcy-form1} and \eqref{eq: Darcy-form2} in primal form:
\begin{equation}
-\nabla \cdot (\kappa \nabla \phi) =f^d \quad  {\rm in}\,\,  \Omega_d.
\label{eq:phi}
\end{equation} 

 The free-flow problem is described by the Stokes equations 
\begin{subequations}
\label{eq:Stokes}
\begin{align}
  -\nu \triangle \bm{u}^s+\nabla p^s&=\bm{f}^s \quad  {\rm in}\,\,\Omega_s, \label{eq: stokes-form1}\\
  \nabla\cdot \bm{u}^s&= 0 \quad  {\rm in}\,\,\Omega_s, \label{eq: stokes-form2}
\end{align}
\end{subequations}
where $\bm{u}^s=( u^s, v^s)$ is the fluid velocity vector,  $p^s$ is the fluid pressure, and $\nu$ is the fluid viscosity.    

Denoting $(\phi, \bm{u}, p)=(p^d, \bm{u}^s, p^s)$, 
Equations \eqref{eq:phi}--\eqref{eq:Stokes} give us the  Stokes--Darcy problem in primal form, with three variables:
\begin{subequations}
\begin{align}
-\kappa \triangle \phi &= f^d \quad  {\rm in}\,\,  \Omega_d, \label{eq:Darcy-phi}\\
  -\nu \triangle \mathbf{u}+\nabla p&=\bm{f}^s\quad  {\rm in}\,\,\Omega_s, \label{eq:Stokes_momentum}\\
  \nabla\cdot \mathbf{u}&= 0\quad  {\rm in}\,\,\Omega_s. \label{eq:incompressibility}
  \end{align}
\end{subequations}
This is an alternative formulation to the one given by Equations \eqref{eq:Darcy} and \eqref{eq:Stokes}, and we will focus from this point onward on this primal form.
The problem is completed by setting interface conditions and imposing boundary conditions. 

The interface conditions can be thought of as the equivalent of a boundary layer through which the velocity changes rapidly.  
The following three interface conditions are often used to couple the Darcy and Stokes equations at the interface $\Gamma$: 
\begin{subequations}
\begin{align}
v = -\kappa \frac{\partial \phi}{\partial y}; \label{eq:masscon}\\
p-\phi  = 2 \nu \frac{\partial v}{\partial y};   \label{eq:balance}\\
u  = \frac{\nu }{\alpha} \left(\frac{\partial u}{\partial y}+ \frac{\partial v}{\partial x}  \right); \label{eq:BJS} 
\end{align}
\label{eq:interface_conditions}
\end{subequations}%
Equation~\eqref{eq:masscon} is a {\em mass conservation condition}, and it guarantees continuity of normal velocity components. Equation \eqref{eq:balance} is a condition on the {\em balance of normal forces}, and it  allows the pressure to be discontinuous across the interface. Finally, \eqref{eq:BJS}, the {\em Beavers-Joseph-Saffman condition}, provides a suitable slip condition on the tangential  velocity. 

The  physical and mathematical properties associated with the interface conditions have been extensively studied in the literature; see, e.g., \cite{tlupova2022domain,hou2019solution}.
A central challenge in the solution of the Stokes--Darcy equations is that the equations governing each domain are fundamentally different. This difficulty is manifested especially when the parameters involved, specifically the viscosity coefficient $\nu$ and permeability constant $\kappa$, differ from each other by a few orders of magnitude. 


\section{Discretization}\label{sec:discretization}
The Marker-and-Cell  scheme \cite{mckee2008mac,eymard2010convergence} is an established and popular discretization technique that has been extensively used in the solution of fluid flow problems \cite{sun2019stability,rui2020mac,shiue2018convergence}.   The components of the velocity and the pressure are discretized at different locations on the grid, in a way that aims at accomplishing numerical stability. Figure~\ref{fig:MAC} shows the location  the discrete variables for \eqref{eq:phi}--\eqref{eq:Stokes}.

\begin{figure}[htp]
\centering
\includegraphics[width=0.4 \textwidth]{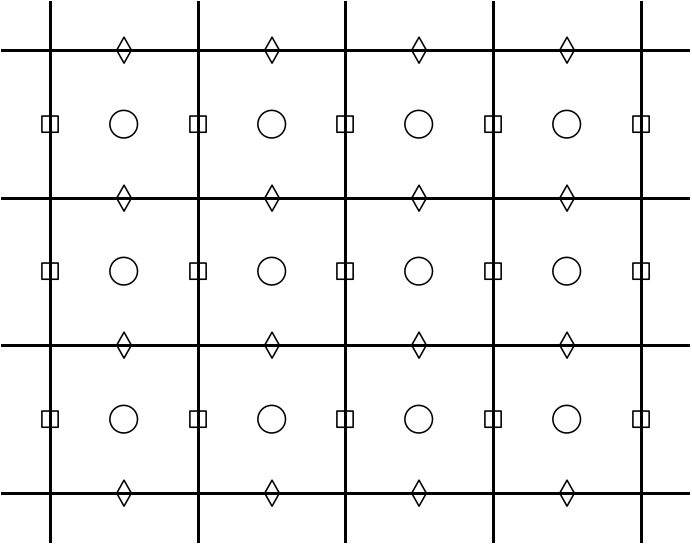}
\hspace{3mm}
\includegraphics[width=0.4 \textwidth]{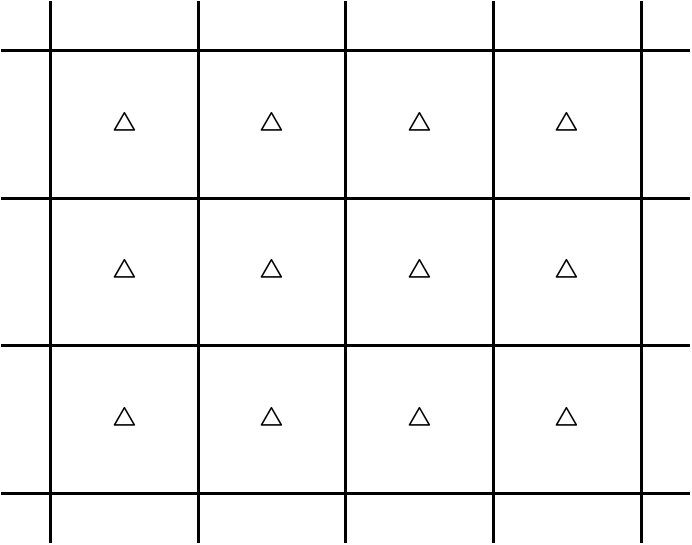}
 \caption{The locations of unknowns on staggered grids. Left: the Stokes variables: $\Box$ -- $u$,  $\lozenge$ -- $v$, $\bigcirc$ --  $p$; Right: the Darcy variable: $\triangle$ --  $\phi$. \label{fig:MAC}}
\end{figure}

The stability and convergence order of the MAC discretization for the Stokes--Darcy equations have been established in the literature. In  \cite{shiue2018convergence}, a stability analysis is performed for the velocity and the pressure, and error estimates are given for uniform grids. 
Let the two subdomains have the same length, $L$, in the $y$ direction.
By \cite[Theorem 4.1]{shiue2018convergence}, if 
\begin{equation}
h \le \min \left\{ \frac{\nu \kappa}{2L},  \frac{2\alpha}{L} \right\},
\label{eq:convcond}
\end{equation}
then first-order convergence is guaranteed. In some of the tests in that paper second-order convergence was in fact experimentally observed. Our discretization follows the discretization of \cite{shiue2018convergence}. In Section \ref{sec:numerical} we provide a brief experimental study of errors. We note that  in \cite{rui2020mac} the authors use a finite volume technique for the tensor format of the fluid operator
near the interface and prove that under the assumption that the solution is sufficiently smooth, second-order
convergence is obtained in the $L_2$-norm for both velocity and pressure of the Stokes and Darcy
flows.

\subsection{Discretization at interior gridpoints for Stokes}\label{sub:Dis-Stokes}
Suppose the Stokes domain is given by $[x^s_{\min},x^s_{\max}] \times [y^s_{\min},y^s_{\max}]$, with $x^s_{\max}-x^s_{\min}=y^s_{\max}-y^s_{\min}$. We consider a uniform mesh with $n+1$ gridpoints in each direction, yielding meshsize $$h=\frac{x^s_{\max}-x^s_{\min}}{n}=\frac{y^s_{\max}-y^s_{\min}}{n}.$$

 For simplicity, throughout we assume that the Stokes and the Darcy domains are both square and are of the same size. We assign double subscripts to the gridpoints, which mark their locations on the grid.
Throughout we will assume that, for a function $f(x,y)$ for example, a value written as $f_{i,j}$ corresponds to an approximation or an exact evaluation of the function  at $x=ih$ and $y=jh$. The same applies for a `half index.'  
Given a double index $(i,j)$, in the MAC configuration the discrete solution for the corresponding  $u$ variable is denoted as $u_{i, j+\frac12},$ and for the corresponding $v$ variable it is denoted as $v_{i+\frac12,j}$.
 Figure~\ref{fig:interior} provides a schematic illustration of the discretization for the interior variables.

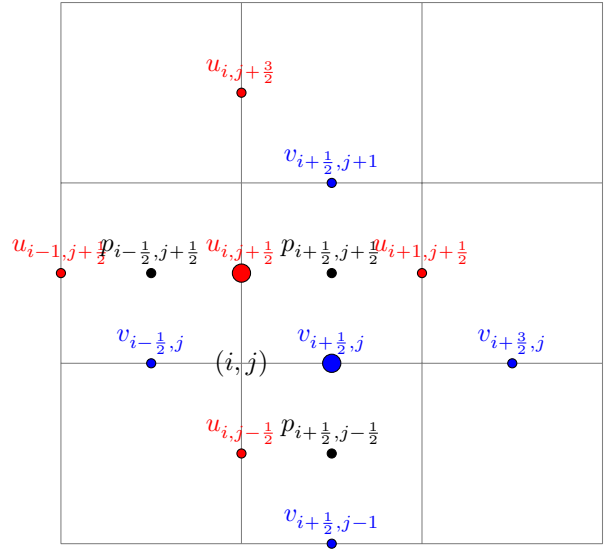
\begin{figure} 
\begin{center}
\begin{tikzpicture}[xscale=0.8,yscale=0.8,>=stealth]
\draw[step=3cm,gray,thin] (-6.001,0) grid (3,9);
\node at (-3,3) {$(i,j)$};
\draw (-3,7.5)[fill=red] circle (0.5ex) node[above] {\color{red}{$u_{i,j+\frac32}$}};
\draw (-3,4.5)[fill=red] circle (1ex) node[above] {\color{red}{$u_{i,j+\frac12}$}};
\draw (-3,1.5)[fill=red] circle (0.5ex) node[above]{\color{red} {$u_{i,j-\frac12}$}};
\draw (0,4.5)[fill=red] circle (0.5ex) node[above] {\color{red}{$u_{i+1,j+\frac12}$}};
\draw (-6,4.5)[fill=red] circle (0.5ex) node[above]{\color{red} {$u_{i-1,j+\frac12}$}};
\draw (1.5,3)[fill=blue] circle (0.5ex) node[above]{\color{blue} {$v_{i+\frac32,j}$}};
\draw (-1.5,6)[fill=blue] circle (0.5ex) node[above]{\color{blue}  {$v_{i+\frac12,j+1}$}};
\draw (-1.5,0)[fill=blue] circle (0.5ex) node[above]{\color{blue}  {$v_{i+\frac12,j-1}$}};
\draw (-1.5,3)[fill=blue] circle (1ex) node[above]{\color{blue}  {$v_{i+\frac12,j}$}};
\draw (-4.5,3)[fill=blue] circle (0.5ex) node[above]{\color{blue}  {$v_{i-\frac12,j}$}};
\draw (-1.5,4.5)[fill=black] circle (0.5ex) node[above]{\color{black}  {$p_{i+\frac12,j+\frac12}$}};
\draw (-1.5,1.5)[fill=black] circle (0.5ex) node[above]{\color{black}  {$p_{i+\frac12,j-\frac12}$}};
\draw (-4.5,4.5)[fill=black] circle (0.5ex) node[above]{\color{black}  {$p_{i-\frac12,j+\frac12}$}};
\end{tikzpicture}
\caption{Discretization of interior gridpoints for the Stokes equations. The gridpoints about which the discretizations are given are marked with bigger circles. The red circles mark $u$ variables and the blue circles mark $v$ variables. The black circles denote pressure. \label{fig:interior}}
\end{center}
\end{figure}

 To further describe the discretization, it is useful to write the Stokes momentum equation \eqref{eq:Stokes_momentum} in scalar form:
 \begin{equation}
 \left\{ \begin{aligned}
  -\nu \left(\frac{\partial^2 u}{\partial x^2} + \frac{\partial^2 u}{\partial y^2} \right) + \frac{\partial p}{\partial x} 
= f^s_1, \\
-\nu \left(\frac{\partial^2 v}{\partial x^2} + \frac{\partial^2 v}{\partial y^2} \right) + \frac{\partial p}{\partial y} = f^s_2,
\end{aligned}  \right.
\label{eq:stokes_momentum}
\end{equation}
where $f^s_i, i=1,  2$ denote the vector-components of $\bm{f}^s$ corresponding to the velocity components $u$ and $v$.
Using centered differences 
for the first and second derivatives, the corresponding discretization for the first equation in \eqref{eq:stokes_momentum}  at gridpoint $(ih, (j+\frac12)h)$ is given by
{\footnotesize
\begin{equation*}
 -\nu \left(\frac{u_{i+1,j+\frac12}+u_{i-1,j+\frac12}+u_{i,j+\frac32}+u_{i,j-\frac12} - 4 u_{i,j+\frac12}}{h^2}  \right) + 
\frac{p_{i+\frac12,j+\frac12}-p_{i-\frac12,j+\frac12}}{h}
 = ({f^s_1})_{i,j+\frac12},
 \end{equation*}
 }%
whereas the discretization for the second equation in \eqref{eq:stokes_momentum}  at gridpoint $((i+\frac12)h,jh)$ is
{\footnotesize
\begin{equation*}
-\nu \left(\frac{v_{i+\frac12,j+1}+v_{i+\frac12,j-1}+v_{i+\frac32,j}+v_{i-\frac12,j} - 4 v_{i+\frac12,j}}{h^2}  \right) +  \frac{p_{i+\frac12,j+\frac12}-p_{i+\frac12,j-\frac12}}{h}
 = ({f^s_2})_{i+\frac12,j}. 
 \end{equation*}
 }%
 Given the staggered grid configuration, we have  $n(n-1)$ gridpoints for $u$ and the same number for $v$, but the internal indexing is different between those two velocity components. For the $u$ variables, the interior gridpoints correspond to $(x_i, y_{j+\frac12}), 1 \le i \le n-1, 0 \le j \le n-1$, and for the $v$ variables the interior gridpoints correspond to $(x_{i+\frac12}, y_j), 0 \le i \le n-1, 1 \le j \le n-1$.

 {\em Boundary conditions.}  If Dirichlet boundary conditions are given, the values for the $u$ gridpoints are prescribed for the vertical boundary points corresponding to $i=0$ and $i=n$. For the horizontal boundary values corresponding to the $u$ variables, since the discrete values closest to the top boundary, i.e., with respect to $j=n$, appear as $u_{i,n-\frac12}, 1 \leq i \leq n-1,$ and are not right on the boundary, we define ghost variables  $u_{i,n+\frac12}, 1 \leq i \leq n-1,$ and use an average $$u_{i,n} = \frac{u_{i,n-\frac12}+u_{i,n+\frac12}}{2}$$ to assign the boundary conditions. It follows that $u_{i,n+\frac12} = 2 u_{i,n} - u_{i,n-\frac12}$, which is used in the discrete Stokes equations for $u_{i,n-\frac12}$. This follows a standard approach; see, for example, \cite{chen2016finite}. The points near $j=0$ are treated separately as part of the interface conditions; see Section~\ref{sec:interface}.
 
As for the $v$ variables, for $j=0$ see Section~\ref{sec:interface}, which describes the interface conditions. For $j=n$ the Dirichlet boundary conditions are prescribed directly. For the discrete values $v_{\frac12,j}$ and $v_{n-\frac12,j}, 1 \leq j \leq n-1,$ we use averages $$v_{0,j} = \frac{v_{-\frac12,j}+v_{\frac12,j}}{2}\quad {\rm and}\quad  v_{n,j} = \frac{v_{n-\frac12,j}+v_{n+\frac12,j}}{2}$$ respectively, from which we extract the ghost variables
$v_{-\frac12,j}$ and $v_{n+\frac12,j}$  and substitute them in the discrete Stokes  equations,  analogously to the $u$ variables.

For example, the discretization of the second equation in~\eqref{eq:stokes_momentum} at gridpoint
$(\frac12 h, h)$ is given by
 \begin{equation*}
 -\nu \,\frac{v_{-\frac{1}{2},1}+ v_{\frac{3}{2},1} + v_{\frac{1}{2},0}+v_{\frac{1}{2},2} -4v_{\frac{1}{2},1}}{h^2} +
 \frac{p_{\frac{1}{2}, \frac{3}{2}}- p_{\frac{1}{2}, \frac{1}{2}}}{h}=(f_2^s)_{\frac{1}{2},1},
 \end{equation*}
 where $v_{-\frac{1}{2},1}$ is a ghost variable, which can be  eliminated by the linear extrapolation $(v_{-\frac{1}{2},1}+ v_{\frac{1}{2},1})/2 = v_{0,1} \equiv v_D(0,h)$, the given Dirichlet boundary condition.  Using this equation to eliminate the ghost variable, we obtain
 \begin{equation}
 -\nu \,\frac{v_{\frac{3}{2},1} + v_{\frac{1}{2},0}+v_{\frac{1}{2},2} -5v_{\frac{1}{2},1}}{h^2} +
 \frac{p_{\frac{1}{2}, \frac{3}{2}}- p_{\frac{1}{2}, \frac{1}{2}}}{h}=(f_2^s)_{\frac{1}{2},1}+\frac{2 \nu v_{0,1}}{h^2}.
 \label{eq:halfhh}
 \end{equation}

\subsection{Discretization at interior gridpoints for Darcy}\label{sub:Dis-Darcy}
The discretization for the Darcy variable, $\phi$, is simpler than the discretization for Stokes. Here we work on the part of the grid in $\Omega^d$.
The Darcy domain is given by $[x^d_{\min},x^d_{\max}] \times [y^d_{\min},y^d_{\max}]$. We assume $x^d_{\max}-x^d_{\min}=y^d_{\max}-y^d_{\min}$ and consider a uniform mesh with meshsize $h$, similarly to the Stokes subdomain: $$h=\frac{x^d_{\max}-x^d_{\min}}{n}=\frac{y^d_{\max}-y^d_{\min}}{n}. $$

We assign negative grid indices for the $y$ variables: $-n \leq j \le 0$. At the gridpoint $((i+\frac12)h,(j+\frac12)h)$, the discretization for \eqref{eq:Darcy-phi} is given by
\begin{equation*}
-\kappa \left(\frac{\phi_{i+\frac12,j-\frac12}+\phi_{i+\frac12,j+\frac32}+\phi_{i+\frac32,j+\frac12}+\phi_{i-\frac12,j+\frac12} - 4 \phi_{i+\frac12,j+\frac12}}{h^2}  \right) 
 = ({f^d})_{i+\frac12,j+\frac12}.
 \end{equation*}

\subsection{Discretization of interface conditions}
\label{sec:interface}
The interface presents a few challenges. We use ghost variable to discretize our variables, as illustrated in Figure~\ref{fig:ghost}. There is a significant difference between the way the $u$ variables and the $v$ variables are handled on the interface. This is because the discrete $v$ variables lie precisely on the interface, whereas the discrete $u$ variables do not. 

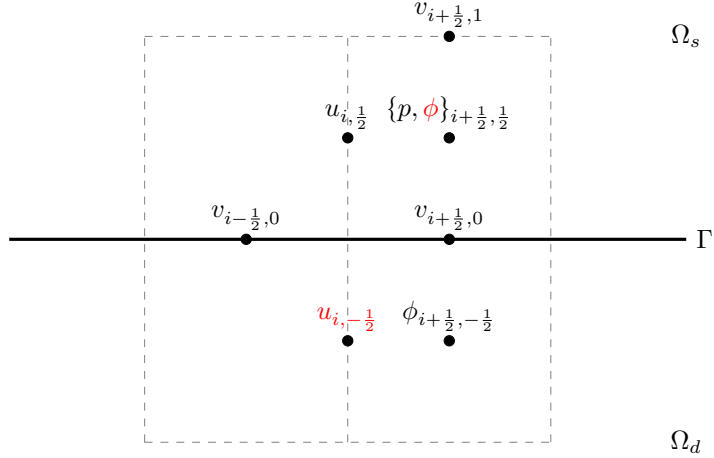
\begin{figure} 
\begin{center}
\begin{tikzpicture}[xscale=0.9,yscale=0.9,>=stealth]
\draw[step=3cm,gray,very thin,dashed] (-3.001,0) grid (3,6);
\draw[black,very thick] (-5,3)node[left]{}  --(5,3) node[right]{$\Gamma$};
  \node at (5,6) {$\Omega_s$};
    \node at (5,0) {$\Omega_d$};
\draw (1.5,3)[fill=black] circle (0.5ex) node[above] {$v_{i+\frac12,0}$};
\draw (1.5,6)[fill=black] circle (0.5ex) node[above] {$v_{i+\frac12,1}$};
\draw (0,4.5)[fill=black] circle (0.5ex) node[above] {$u_{i,\frac12}$};
\draw (0,1.5)[fill=black] circle (0.5ex) node[above] {{\color{red}$u_{i,-\frac12}$}};
\draw (-1.5,3)[fill=black] circle (0.5ex) node[above] {$v_{i-\frac12,0}$};
\draw (1.5,4.5)[fill=black] circle (0.5ex) node[above] {$\{p,{\color{red}\phi}\}_{i+\frac12,\frac12}$};
\draw (1.5,1.5)[fill=black] circle (0.5ex) node[above] {$\phi_{i+\frac12,-\frac12}$};
\end{tikzpicture}
\caption{Discretization of the variables near the interface. The ghost variables that are to be eliminated are marked in red. \label{fig:ghost}}
\end{center}
\end{figure}

Following \cite{shiue2018convergence}, the interface conditions are discretized as follows. For $1 \leq i \leq n-1$:
\begin{itemize}
\item mass conservation, $ v = -\kappa \,\, \frac{\partial \phi}{\partial y}$: 
\begin{equation}
  v_{i+\frac12,0} = -\kappa  \frac{\phi_{i+\frac12,\frac12} - \phi_{i+\frac12,-\frac12}}{h}  \label{eq:dis-mass-conservation} 
\end{equation} 

\item balance of normal forces, $p-\phi  = 2 \nu  \frac{\partial v}{\partial y}$:  
\begin{equation}
  p_{i+\frac12,\frac12}-\phi_{i+\frac12,-\frac12} = 2 \nu \, \frac{v_{i+\frac12,1}-v_{i+\frac12,0}}{h} \label{eq:dis-balance-norm}
\end{equation} 

\item Beavers-Joseph-Saffman (BJS) condition, $u  = \frac{\nu }{\alpha} \left(\frac{\partial u}{\partial y}+ \frac{\partial v}{\partial x}  \right)$:
\begin{equation}
  \frac{u_{i,\frac12} + u_{i,-\frac12}}{2} = \frac{\nu}{\alpha} \left( \frac{u_{i,\frac12}-u_{i,-\frac12}}{h}+\frac{v_{i+\frac12,0}-v_{i-\frac12,0}}{h} \right)  \label{eq:dis-BJS}
\end{equation}

\end{itemize}

Equations \eqref{eq:dis-mass-conservation}--\eqref{eq:dis-BJS} are coupled with the discretized Stokes equations  and the discretized Darcy equations.  The discretized Darcy equations for $\phi_{i+\frac12,-\frac12}$  involve the ghost values,  $\phi_{i+\frac12,\frac12}$, which can be eliminated  using \eqref{eq:dis-mass-conservation}. 

The discretized equations for interface variables $v_{i+\frac12,0}$ are formed using \eqref{eq:dis-balance-norm}.
The discretized  Stokes equations for the $u_{i,\frac12}$ variables  involve the ghost values,  $u_{i,-\frac12}$, which can be eliminated  using \eqref{eq:dis-BJS}. 
 
\subsection{The linear system}

Putting together the equations for the interior gridpoints and the interface conditions, and incorporating boundary conditions, we obtain a double saddle-point  system of the form
\begin{equation}\label{eq:coupled-Darcy-stokes-system}
\begin{pmatrix}
 A_d &      -G^T    & 0  \\ 
G   &       A_s        & B ^T \\
0 &      B               & 0  
 \end{pmatrix}
 \begin{pmatrix}
\phi_h \\
 \bm{u}_h\\
 p_h
  \end{pmatrix} = 
 \begin{pmatrix}
g_1\\
 \bm{g}_2\\
 g_3
 \end{pmatrix},
\end{equation} 
where $A_d$ corresponds to  $-\kappa \triangle$  for the Darcy  equation and $A_s (\ne A_s^T)$ is  the discretization of $-\nu \triangle$ for the Stokes equations coupled with the  discretized interface conditions. The last block row in \eqref{eq:coupled-Darcy-stokes-system} corresponds to  the (negated) divergence-free condition.
Due to the boundary and interface conditions, the coefficient matrix in \eqref{eq:coupled-Darcy-stokes-system} is nonsymmetric.
 Double saddle-point systems of a similar form have been extensively studied recently \cite{bradley2021eigenvalue,holter2021robust,cai2009preconditioning}, but the focus of spectral studies has  been on symmetric instances. In this paper we offer new insights on the nonsymmetric case.
 
 The linear system \eqref{eq:coupled-Darcy-stokes-system} has  $4n^2-n$ unknowns, and we have
$A_d \in \mathbb{R}^{n^2 \times n^2}$,  $A_s \in \mathbb{R}^{(2n^2-n)\times (2n^2-n)}$,   $G \in \mathbb{R}^{(2n^2-n)\times n^2}$,   and $B \in \mathbb{R}^{n^2\times n^2}$. In the sequel we describe the structure of the submatrices of \eqref{eq:coupled-Darcy-stokes-system}. To avoid ambiguity when it may arise, when necessary we attach subscripts to identity matrices to indicate their sizes.

 \subsubsection*{The matrix $A_d$}
The matrix $A_d$ can be naturally partitioned as a $2\times 2$ block matrix having the following structure:
\begin{equation}
A_d = \begin{pmatrix}
 A_{d,11}  &   A_{d,12}    \\ 
 A_{d,21}  &    A_{d,22}  
 \end{pmatrix}, \quad   A_d = A_d^T, \quad 
A_{d,12}=A_{d,21}^T,
\label{eq:A_d}
\end{equation} 
where $A_{d,11} \in {\mathbb R}^{(n^2-n) \times (n^2-n)}, \ A_{d,21} \in {\mathbb R}^{n \times (n^2-n)},$  $A_{d,22} \in {\mathbb R}^{n \times n}$,  and  
$$A_{d,21} =-\frac{\kappa}{h^2}(0 \quad  I_n). $$
The second block row of  $A_{d},$ namely $(A_{d,21}\,\, A_{d,22})$, corresponds to the discrete  $n$ equations for $\phi$ near the interface $\Gamma$, and it is coupled with the discrete interface variables $v$, which appear in $G^T$; see \eqref{eq:coupled-Darcy-stokes-system}.

 \subsubsection*{The matrix $A_s$}
The matrix $A_s$ is a $3\times 3$ block matrix with the structure
\begin{equation} 
A_s =\begin{pmatrix}
 A_{11} &      A_{12}   & 0  \\ 
0   &       A_{22}       & A_{23} \\
0 &      A_{32}               & A_{33}
 \end{pmatrix};
 \label{eq:As}
\end{equation} 
Figure \ref{fig:As} depicts the dimensions of the blocks.
 
The matrix $A_{12}$ is $(n^2-n) \times n$, as can be inferred from Figure \ref{fig:As}, and  it is mostly zero. It is comprised of an $(n-1) \times n$ upper bidiagonal block stacked on top of an $(n^2-2n+1) \times n$ zero block. The bidiagonal block is given by $c \cdot {\rm bidiag}[1,-1]$, where $c=\frac{2\nu^2}{h^2 (2\nu+h \alpha)}$. This matrix represents the discretization of the discrete function values $u_{i,\frac12}$, $1 \leq i \leq n-1$, which interact with the interface variables $v_{i+\frac12,0}$, using \eqref{eq:dis-BJS}. 

The matrix $A_{22}$, which corresponds to the interface $v$ variables, has dimensions $n \times n$ and a simple structure: it is equal to a scaled identity matrix with $\frac{2 \nu}{h^2}$. 

The blocks of $A_s$ satisfy
 $A_{11}=A_{11}^T, A_{22}=A_{22}^T, A_{33}=A_{33}^T$, and 
 \begin{equation*}
A_{22} =  \frac{2\nu}{h^2} I_n, \quad A_{23} = (-A_{22},\,\, 0), \quad A_{32}=\frac{1}{2}A_{23}^T.
\end{equation*} 
Notice that while both $A_{11}$ and $A_{33}$ are $(n^2-n) \times (n^2-n)$, their internal block structures are different, due to the staggered grid. The matrix $A_{11}$ (which corresponds to the $u$ variables) is block tridiagonal with $n$ blocks of dimensions $(n-1) \times (n-1)$, whereas $A_{33}$ (which corresponds to the $v$ variables) is block tridiagonal with $n-1$ blocks of dimensions $n \times n$ each.

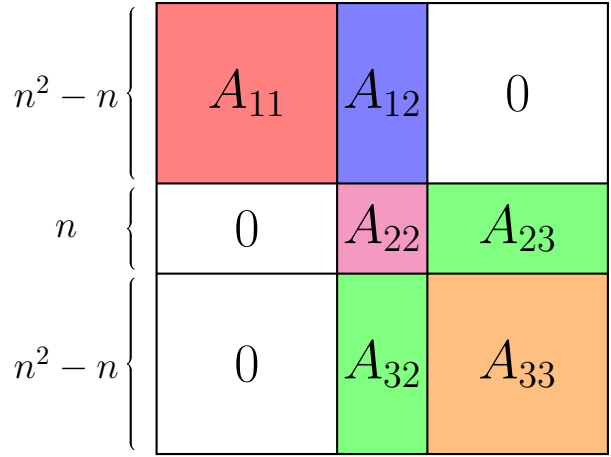
\begin{figure}
\begin{center}
\begin{tikzpicture}[thick, scale=0.6]
\draw[step=1cm,white] (0,0) grid (10,10);
\fill[red!50] (0,6) rectangle (4,10);
\fill[blue!50] (4,6) rectangle (6,10);
\fill[green!50] (4,0) rectangle (6,4);
\fill[green!50] (6,4) rectangle (10,6);
\fill[magenta!50] (4,4) rectangle (6,6);
\fill[orange!50] (6,4) rectangle (10,0);
\draw (2,8) node{{\huge $A_{11}$}};
\draw (2,5) node{{\huge $0$}};
\draw (5,8) node{{\huge $A_{12}$}};
\draw (2,2) node{{\huge $0$}};
\draw (8,8) node{{\huge $0$}};
\draw (5,2) node{{\huge $A_{32}$}};
\draw (8,5) node{{\huge $A_{23}$}};
\draw (5,5) node{{\huge $A_{22}$}};
\draw (8,2) node{{\huge $A_{33}$}};
\draw (-2,8) node{{\Large $n^2-n$}};
\draw (-2,5) node{{\Large $n$}};
\draw (-2,2) node{{\Large $n^2-n$}};
\draw (0,0) rectangle (10,10);
\draw (0,4) -- (10,4);
\draw (0,6) -- (10,6);
\draw (4,0) -- (4,10);
\draw (6,0) -- (6,10);
\draw [decorate, decoration = {calligraphic brace}] (-0.5,6.1) --  (-0.5,9.9);
\draw [decorate, decoration = {calligraphic brace}] (-0.5,4.1) --  (-0.5,5.9);
\draw [decorate, decoration = {calligraphic brace}] (-0.5,0.1) --  (-0.5,3.9);
\end{tikzpicture}
\caption{Block structure of $A_s$. \label{fig:As}}
\end{center}
\end{figure}

\subsubsection*{The coupling matrix $G$} 
 The equations for the $u_{i,\frac12}$ variables are coupled with the discrete interface variables $v_{i+\frac12,0}$, which are represented by the matrix $G$ shown in \eqref{eq:coupled-Darcy-stokes-system}. 
$G^T$ is a $2\times 3$ block matrix with the following attractively simple structure:
\begin{equation} \label{eq:G-form}
G^T =\begin{pmatrix}
0 &     0  & 0  \\ 
0   &       -I_n/h       & 0
 \end{pmatrix}.
\end{equation} 
The nonzero block arises from the discretization of $\phi_{i+\frac12,-\frac12}$, using \eqref{eq:dis-mass-conservation}.

\subsubsection*{The matrix $B$}
The matrix $B$ is a standard discrete divergence operator given by 
\begin{equation}\label{eq:B-three-block}
B =\begin{pmatrix}
B_x  & B_0  &B_y
 \end{pmatrix} \in \mathbb{R}^{n^2 \times (2n^2-n)}, \quad B_0=\begin{pmatrix} I_n/h  \\ 0 \end{pmatrix}  \in \mathbb{R}^{n^2 \times n}.
\end{equation}

\subsubsection*{Solvability conditions}
For simplicity, we assume that pure Dirichlet  boundary conditions are imposed, that is:
\begin{subequations}
\label{eq:BCs}
\begin{align}
  \bm{u}^s &=g_D^s \quad  {\rm on}\,\,\partial\Omega_s, \nonumber \\
  \phi&= g_D^d \quad  {\rm on}\,\,\partial\Omega_d. \nonumber
\end{align}
\end{subequations}
The pressure is assumed to satisfy the condition
\begin{equation*}
\int_{\Omega_s} p^s \, dx =0,
\end{equation*} 
which yields a unique solution for the discrete system \eqref{eq:coupled-Darcy-stokes-system}. 

Neumann or mixed boundary conditions are also commonly considered; see, for example, \cite{luo2017uzawa,rui2020mac,shiue2018convergence} and the references therein.

\subsection{Properties of the matrices}
\label{sec:properties}

Let us  rewrite the linear system \eqref{eq:coupled-Darcy-stokes-system} in a form that symmetrizes the off-diagonal blocks:
\begin{equation*}
\begin{pmatrix}
 A_d &      G^T    & 0  \\ 
G   &      - A_s        & B ^T \\
0 &      B               & 0  
 \end{pmatrix}
 \begin{pmatrix}
\phi_h \\
 -\bm{u}_h \\
 p_h
  \end{pmatrix} = 
 \begin{pmatrix}
 g_1\\
 \bm{g}_2\\
- g_3
 \end{pmatrix}.
\end{equation*} 
Let
\begin{equation} 
\mathcal{K}=\begin{pmatrix}
 A_d &      G^T    & 0  \\ 
G   &       -A_s        & B ^T \\
0 &      B               & 0  
 \end{pmatrix}.
 \label{eq:K}
 \end{equation}
  The blocks of  $\mathcal{K}$ satisfy a few useful properties.
\begin{enumerate}
\item $A_s$ is nonsymmetric and positive definite. 
\item $(G \quad B^T)$ has a one-dimensional null space spanned by an all-ones vector of size $2n^2$. 
\item $B$ has full rank.
\item If we consider Neumann boundary conditions for the Darcy problem, then $A_d$ is symmetric positive semidefinite with a one-dimensional null space spanned by all-ones vector.     $\mathcal{K}$ is nonsymmetric and singular with a one-dimensional null space spanned by  $\begin{pmatrix} e \\ 0 \\ e \end{pmatrix}$, where $e$ is the vector of all ones of length $n^2$ and $0$ is the zero vector of length $2n^2-n$. 
\item If we consider Dirichlet boundary conditions for the Darcy problem, then $A_d$ is symmetric   positive definite,  and   $\mathcal{K}$ 
is nonsymmetric and nonsingular. 
\end{enumerate}

\begin{lemma} All  eigenvalues of $A_s$, which  represents the Stokes equations and interface equations and is given in \eqref{eq:As}, are positive. 
\end{lemma}
\begin{proof} The eigenvalues of $A_s$ are a union of the eigenvalues of $A_{11}$ and
\begin{equation*}
E =\begin{pmatrix}
 A_{22}       & A_{23} \\
 A_{32}           & A_{33}
 \end{pmatrix} = \begin{pmatrix}
 A_{22}       &2 A_{32}^T \\
 A_{32}           & A_{33}
 \end{pmatrix}.
 \end{equation*}
 The matrix $E$ is symmetrizable by a diagonal matrix
 $\tilde{D} =\begin{pmatrix}
 I_n       & 0 \\
0           & \sqrt{2} I_{n^2-n}
 \end{pmatrix},
 $ and therefore its eigenvalues are real.
Since $A_{11}$ is symmetric and diagonally dominant, its eigenvalues are positive.

Let $\tilde{A}_{32}= \sqrt{2} A_{32}$. The block $L D L^T$  decomposition of $\tilde{E}=\tilde{D} E \tilde{D}^{-1}$ is 
 \begin{equation*}
\tilde{E}=  \begin{pmatrix}
 A_{22}        & \tilde{A}_{32}^T\\
\tilde{A}_{32}          &A_{33} 
 \end{pmatrix} 
 =\begin{pmatrix}
 I_n     &  0\\
\tilde{A}_{32}A_{22} ^{-1}   & I_{n^2-n}
 \end{pmatrix} \!
 \begin{pmatrix}
 A_{22}       &  0\\
   0   &A_{33} -\tilde{A}_{32}A_{22} ^{-1}\tilde{A}_{32}^T
 \end{pmatrix} \!
  \begin{pmatrix}
 I_n       &A_{22} ^{-1}  \tilde{A}_{32}^T \\
0  & I_{n^2-n}
 \end{pmatrix} .
\end{equation*}
A simple calculation shows that 
\begin{align*}
A_{33} -\tilde{A}_{32}A_{22} ^{-1}\tilde{A}_{32}^T & = A_{33} - \frac{1}{2}(-A_{22}\,\, 0)^TA_{22} ^{-1}(-{A}_{22}\,\,0) \\
 &= A_{33} - \begin{pmatrix} \frac{\nu}{h^2}I_n & 0 \\ 0 & 0 \end{pmatrix}.
\end{align*} 
Thus, the above matrix is the same as $A_{33}$ except the top left $n \times n$ block, and we now discuss the structure of that specific block of $A_{33}$. 

The first and $n$th rows of $A_{33}$ have three nonzero elements $[-\nu/h^2, 5\nu/h^2,-\nu/h^2]$, where the value 5 is due to Dirichlet boundary conditions; see \eqref{eq:halfhh}. 
Rows  2 to $n-1$ have four nonzero elements $[-\nu/h^2, 4\nu/h^2,-\nu/h^2, -\nu/h^2]$, where the positive values  are located at the diagonal position and we have diagonal dominance here.
It follows that all  eigenvalues of $A_s$ are positive, as required.
\end{proof} 

Next, we state a rank property of $B$, which will be used later in our spectral analysis. 
\begin{lemma}\label{lem:rank-prop-B}
Define 
\begin{equation}
\bar{B} =\begin{pmatrix}
B_x     &B_y
 \end{pmatrix} \in \mathbb{R}^{n^2 \times m_2},
 \label{eq:Bbar}
\end{equation}
where $m_2=(2n^2-n)-n=2n^2-2n$. Then, ${\rm rank}(\bar{B})=n^2-1$ and the nullity of $\bar{B}$ is $(n-1)^2$.
\end{lemma}


\section{Block preconditioners}\label{sec:preconditioner}

Block factorizations of the double saddle-point matrix $\mathcal{K}$ defined in \eqref{eq:K}  motivate the derivation of potential preconditioners. We write 
 \begin{eqnarray}
 \begin{aligned}
 \begin{pmatrix}
 A_d &      G^T    & 0  \\ 
G   &      - A_s        & B^T \\
0 &      B               & 0  
 \end{pmatrix} & =  
 \underbrace{\begin{pmatrix}
 I &     0  & 0  \\ 
GA_d^{-1}   &     I       &0 \\
0 &      -B S_1^{-1}             & I 
 \end{pmatrix}}_L
\underbrace{ \begin{pmatrix}
 A_d &     0  & 0  \\ 
0  &    -S_1      &0 \\
0 &     0             &S_2
 \end{pmatrix}}_D
\underbrace{
 \begin{pmatrix}
 I &    A_d^{-1}G^T     & 0  \\ 
 0  &     I                  & -S_1^{-1 }B^T \\
0 &     0          & I 
 \end{pmatrix}}_U \\
 & = 
 \underbrace{
 \begin{pmatrix}
 A_d &     0  & 0  \\ 
G  &    -S_1      &0 \\
0 &     B             &S_2
 \end{pmatrix}}_{LD}
 \underbrace{
 \begin{pmatrix}
 I &    A_d^{-1}G^T     & 0  \\ 
 0  &     I                  & -S_1^{-1 }B^T \\
0 &     0          & I 
 \end{pmatrix}}_U,
 \end{aligned}
 \label{eq:block_K}
\end{eqnarray}
where
\begin{equation}
S_1= A_s+GA_d^{-1}G^T
\label{eq:S1} 
\end{equation} 
and
\begin{equation}
S_2 = B S_1^{-1}B^T
\label{eq:S2}
\end{equation}
are Schur complements. 

In \eqref{eq:block_K} we have written two  forms of factorizations. The first is a block LDU factorization with $L$  a unit lower triangular matrix and $D$ a block diagonal one, and the second  is a block decomposition where the lower block-triangular matrix is simply the product of $LD$ in the LDU block factorization. 
We use these forms to consider  block preconditioners. The Appendix provides additional options.

Ideal preconditioners we consider and analyze are:
\begin{equation*}
 \mathcal{M}_1 =\begin{pmatrix}
 A_d &     0       & 0  \\ 
0    &    S_1    & 0 \\
0     &     0       & S_2
 \end{pmatrix}, 
 \quad   \mathcal{M}_2 =\begin{pmatrix}
 A_d &     0       & 0  \\ 
G     &    S_1    & 0 \\
0     &     0       & S_2
 \end{pmatrix}, 
\quad
 \mathcal{M}_3= \begin{pmatrix}
 A_d &     0       & 0  \\ wh
G     &    -S_1    & 0 \\
0     &     B        & S_2
 \end{pmatrix}. 
\end{equation*}
The choice of $\mathcal{M}_1$  arises from the matrix $D$ of the LDU factorization of $\mathcal{K}$; the signs are rearranged so that $\mathcal{M}_1$ is symmetric positive definite. 

Since  $\mathcal{K}$ is nonsymmetric and  $G$ is an interface matrix that contains important physical information on the coupling effect between the Stokes and Darcy equations, it makes sense to  consider block triangular preconditioners as well. The choice of $\mathcal{M}_2$ amounts to a relatively modest revision of $\mathcal{M}_1$, where the  interface matrix $G$ is added as the (2,1) block. The matrix
 $\mathcal{M}_3$ is equal to $LD$  in \eqref{eq:block_K}.

Recall from Section~\ref{sec:properties} that if Neumann boundary conditions are considered for the Darcy problem, then the matrix $A_d$ is positive semidefinite with a one-dimensional null space spanned by the all-ones vector. The singularity presents a challenge for the design of preconditioners, and we do not further pursue this scenario in this paper.  We therefore focus on Dirichlet boundary conditions, for which $A_d$ is symmetric positive definite and the Schur complements are well defined.


\subsection{Spectral analysis}
\label{sec:spectral}

There is an increasing body of literature on symmetric double saddle-point systems. Block diagonal preconditioners have been extensively analyzed \cite{beik18, beigl20,bradley2021eigenvalue,cai21,cai2009preconditioning,hm19,pearson21,pearson21b,sogn18}, including bounds on the eigenvalues and theoretical observations on their algebraic multiplicities.  
The double saddle-point matrix considered in this paper bears similarities, but it has a few distinct features, including nonsymmetry and the fact that the interface conditions are associated with only $n$ unknowns and the corresponding nonzero blocks are $n \times n$, and the other blocks of $\mathcal{K}$  are quadratic in $n$.

\begin{theorem} \label{thm:eigs-M1K}
The matrix $ \mathcal{M}_1^{-1}\mathcal{K}$ has the following eigenvalues and algebraic multiplicities: 
\begin{enumerate}[(i)]
\item  $1$ with multiplicity $n^2-n$;
\item  $-1$ with multiplicity $(n-1)^2$;
\item   $\frac{-1\pm\sqrt{5}}{2}$ with multiplicity $n^2-n$ for each.
\end{enumerate}
In addition:
\begin{enumerate}[(i)]
\item[(a)]  At most $n$ eigenvalues are larger than $1$.
\item[(b)]  At most  $n$ eigenvalues are located at $(0, 1)\setminus \left\{\frac{-1+\sqrt{5}}{2}\right\}$.
\end{enumerate} 
\end{theorem}
\begin{proof}
By direct calculation,
\begin{equation*}
  \mathcal{M}_1^{-1}\mathcal{K}  =\begin{pmatrix}
 I &     A_d^{-1} G^T      & 0  \\ 
S_1^{-1}G    &    -S_1^{-1}A_s   & S_1^{-1}B^T \\
0     &    S_2^{-1} B      &0
\end{pmatrix}.
\end{equation*}
Let $\begin{pmatrix} x^T &  y^T & z^T \end{pmatrix}^T$ be an eigenvector of $ \mathcal{M}_1^{-1}\mathcal{K}$ associated with eigenvalue $\lambda$, that is
\begin{equation*} 
\begin{pmatrix}
 I &     A_d^{-1} G^T      & 0  \\ 
S_1^{-1}G    &    -S_1^{-1}A_s   & S_1^{-1}B^T \\
0     &    S_2^{-1} B      &0
\end{pmatrix} \begin{pmatrix}
 x \\ y\\z
  \end{pmatrix}
  =\lambda  \begin{pmatrix}
 x \\ y\\z
  \end{pmatrix}.
\end{equation*}
We thus have
\begin{subequations}
\label{eq:eigM1}
\begin{align}
x +A_d^{-1}G^T y &=  \lambda x, \label{eq:eig-express-1-P3}\\
S_1^{-1}G x -S_1^{-1}A_s y+S_1^{-1}B^T z&=\lambda y, \label{eq:eig-express-2-P3}\\
 ( B S_1^{-1}B^T)^{-1} B y&= \lambda z. \label{eq:eig-express-3-P3}
\end{align}
\end{subequations}
(i) {\underline {eigenvalue $\lambda=1$}}: When $y=z=0$, \eqref{eq:eigM1} is reduced to
\begin{align*}
x &=  \lambda x,\\
S_1^{-1}Gx &=0, 
\end{align*}
which means that $\lambda=1$ is an eigenvalue of $ \mathcal{M}_1^{-1}\mathcal{K}$  with $G x=0$. Since the null space of $G$ has dimension $n^2-n$, see \eqref{eq:G-form}, $\lambda=1$ is an eigenvalue with multiplicity $n^2-n$.\\
 
 \noindent
(ii) {\underline {eigenvalue $\lambda=-1$}}:  If $x=z=0$, then \eqref{eq:eigM1} is reduced to
\begin{subequations}
\label{eq:eigM1lam1}
\begin{align}
A_d^{-1}G^T y &= 0, \label{eq:eig-express-2-P3-1} \\
-S_1^{-1}A_s y &=\lambda y, \label{eq:eig-express-2-P3-2} \\
 By & =0. \label{eq:eig-express-2-P3-3}
\end{align}
\end{subequations}
We have $A_s=S_1-GA_d^{-1}G^T$. Using \eqref{eq:eig-express-2-P3-1}, we rewrite \eqref{eq:eig-express-2-P3-2} as
\begin{equation*}
-S_1^{-1}(S_1-GA_d^{-1}G^T)y  =-y +0 =\lambda y,
\end{equation*} 
which means that $\lambda=-1$. Next we prove that such $y \ne 0$ exists.   From  \eqref{eq:eig-express-2-P3-1}  and \eqref{eq:G-form}, we see that $y$ has the following structure
\begin{equation*}
y= \begin{pmatrix}
y_1 \\   0  \\  y_2
 \end{pmatrix},
\end{equation*} 
where $y_1$ and $y_2$ can have any value, as long as they are not simultaneously zero. Now, we consider \eqref{eq:eig-express-2-P3-3}. Then, $y_1, y_2$ satisfy 
$\bar{B}\begin{pmatrix} y_1^T &  y_2^T \end{pmatrix}^T=0$  (see \eqref{eq:Bbar}). From  Lemma \ref{lem:rank-prop-B} we know that the nullity of $\bar{B}$ is $(n-1)^2$, which  is the multiplicity of the  eigenvalue $-1$. \\ 
 
 \noindent
(iii) {\underline {eigenvalues $\lambda=\frac{-1\pm\sqrt{5}}{2}$}}: If $x=0, y\neq 0, z \neq 0$, then \eqref{eq:eigM1} is reduced to 
\begin{subequations}
\label{eq:eigM1golden}
\begin{align}
 A_d^{-1}G^T y &=0, \label{eq:eig-express-3-P3-1}\\
  -S_1^{-1}A_s y +S_1^{-1}B^T z&=\lambda y, \label{eq:eig-express-3-P3-2}\\
 ( B S_1^{-1}B^T)^{-1} B y&= \lambda z. \label{eq:eig-express-3-P3-3}
\end{align}
\end{subequations}
Using  $A_s=S_1-GA_d^{-1}G^T$ and \eqref{eq:eig-express-3-P3-1}, we rewrite \eqref{eq:eig-express-3-P3-2} as
\begin{equation*}
-S_1^{-1}(S_1-GA_d^{-1}G^T) y+S_1^{-1}B^T z =-y+S_1^{-1}B^T z=\lambda y, 
\end{equation*} 
which gives  $y=\frac{1}{1+\lambda} S_1^{-1}B^T z$.  Substituting  $y$ into \eqref{eq:eig-express-3-P3-3} gives
\begin{equation*}
 (B S_1^{-1}B^T)^{-1} B y = \frac{1}{1+\lambda} (B S_1^{-1}B^T)^{-1} B S_1^{-1}B^T z= \frac{1}{1+\lambda} z= \lambda z. 
\end{equation*} 
It follows that $\frac{1}{1+\lambda}=\lambda$. Then we have $\lambda=\frac{-1\pm\sqrt{5}}{2}$.   From \eqref{eq:eig-express-3-P3-1}  we have $G^T y =0$, which means  we have a set of  $n^2-n$ linearly independent vectors $y$ here. It follows that the pair of eigenvalues  $\frac{-1\pm\sqrt{5}}{2}$ have  multiplicity  $n^2-n$ each.
Next, we prove that the number of eigenvalues that satisfy $\lambda>1$ is at most $n$. From \eqref{eq:eig-express-1-P3}, we have
\begin{equation}\label{eq:x=formy}
x=\frac{1}{\lambda-1}A_d^{-1}G^Ty.
\end{equation} 
We claim that $G^Ty \neq 0$. This can be shown by contradiction, as follows. If $G^Ty =0$, from \eqref{eq:eig-express-1-P3}, we would have $x=0$. At this point, if $z=0$, then from the proof of (ii) it would follow that  $\lambda=-1$, which contradicts our assumption that $\lambda>1$. So $z\neq 0$. If $y \neq 0$, from the proof of  (iii), we would have  $\lambda=\frac{-1\pm\sqrt{5}}{2}$, which contradicts our assumption that $\lambda>1$.  So $y=0$. However, this leads to $z=0$, which is a contradiction. Thus,   $G^Ty \neq 0$, that is, $y \not\in {\rm ker}(G^T)$. Since  rank$(G^T)=n$, there are at most $n$ such linearly independent vectors $y$.  From \eqref{eq:eig-express-3-P3}, we have
 \begin{equation*}
z=(\lambda B S_1^{-1}B^T)^{-1}By.
\end{equation*} 
So the space spanned by the eigenvectors  $\begin{pmatrix} x^T &  y^T & z^T \end{pmatrix}^T$ has dimension at most $n$. 

Next, we claim that there are $n^2$ eigenvalues in the interval $(0,1)$. Substituting \eqref{eq:x=formy} into \eqref{eq:eig-express-2-P3} and solving for $y$ gives
\begin{equation*}
y =\left(\frac{1}{1-\lambda} GA_d^{-1}G^T +\lambda S_1+A_s \right)^{-1}B^Tz
\end{equation*} 
Since $B^T$ is full rank,  it follows that $z\neq 0$; otherwise, $y=x=0$. Thus, $z$ is in the range of $B^T$. Note that $B^T$ has rank $n^2$.
The space spanned by the eigenvectors $\begin{pmatrix} x^T &  y^T & z^T \end{pmatrix}^T$ has dimension at most $n^2$.  
 From (iii), we know that $\frac{-1+\sqrt{5}}{2}$ has multiplicity $n^2-n$, so the number of eigenvalues in $(0,1)\setminus \{ \frac{-1+\sqrt{5}}{2} \}$ is at most $n^2-(n^2-n)=n$.
%
\end{proof}
\begin{remark} For symmetric block diagonal preconditioners applied to symmetric double saddle-point systems, spectral studies provide results on the boundedness away from zero of all the eigenvalues of the preconditioned matrices; see, e.g., \cite[Theorem 3.3]{bradley2021eigenvalue}. In Theorem~\ref{thm:eigs-M1K} we do not know the location of  $2n-1$ of the $4n^2-n$ eigenvalues. 
\end{remark}

\begin{theorem} \label{thm:M2K}
The eigenvalues of $ \mathcal{M}_2^{-1}\mathcal{K}$ are 
\begin{enumerate}[(i)]
\item $1$ with multiplicity $n^2$;
\item  $-1$ with multiplicity $n^2-n$; 
\item $\frac{-1\pm \sqrt{5}}{2}$ with multiplicities $n^2$ each.
\end{enumerate}
\end{theorem}
\begin{proof}
It can be shown that
\begin{equation*}
 \mathcal{M}_2^{-1} =\begin{pmatrix}
 A_d^{-1} &     0       & 0  \\ 
-S_1^{-1} G A_d^{-1}      &    S_1^{-1}    & 0 \\
0     &     0       & S_2^{-1} 
\end{pmatrix},
\end{equation*}
 and  it follows that
 \begin{equation*}
 \mathcal{M}_2^{-1}\mathcal{K} =\begin{pmatrix}
 I &     A_d^{-1} G^T      & 0  \\ 
0    &    -I    & S_1^{-1}B^T \\
0     &    S_2^{-1} B      &0
 \end{pmatrix}.
\end{equation*}
Let $\begin{pmatrix} x^T &  y^T & z^T \end{pmatrix}^T$ be an eigenvector of $ \mathcal{M}_2^{-1}\mathcal{K}$ associated with eigenvalue $\lambda$, that is,
 \begin{equation*} 
\begin{pmatrix}
 I &     A_d^{-1} G^T      & 0  \\ 
0    &    -I    & S_1^{-1}B^T \\
0     &    S_2^{-1} B      &0
 \end{pmatrix}
 \begin{pmatrix}
 x \\ y\\z
  \end{pmatrix}
  =\lambda   \begin{pmatrix}
 x \\ y\\z
  \end{pmatrix}.
\end{equation*}
We rewrite the above  as
\begin{subequations}
\begin{align}
x +A_d^{-1}G^Ty &=  \lambda x, \label{eq:eig-express-1}\\
-y+S_1^{-1}B^Tz&=\lambda y, \label{eq:eig-express-2}\\
 ( B S_1^{-1}B^T)^{-1} B y&= \lambda z. \label{eq:eig-express-3}
\end{align}
\end{subequations}
It is obvious that   $\begin{pmatrix} x^T &  y^T & z^T \end{pmatrix}^T=\begin{pmatrix} x^T  &  0  & 0 \end{pmatrix}^T$ where $x \neq 0$ is an eigenvector of $ \mathcal{M}_2^{-1}\mathcal{K}$ with  $\lambda =1$.  Since $x\in \mathbb{R}^{n^2\times 1}$,  we have that $\lambda=1$ is an eigenvalue with multiplicity $n^2$.

If $\lambda=-1$ and $y \neq 0$, from \eqref{eq:eig-express-2} we have $S_1^{-1}B^Tz=0$. It follows that $B^T z=0$. Since $B^T$ has full rank, $z=0$.
From \eqref{eq:eig-express-3}, we have $By=0$. Since  $B \in \mathbb{R}^{n^2 \times (2n^2-n)}$ has rank $n^2$, the  null space of  $B$  has dimension $2n^2-n-n^2=n^2-n$.

If $\lambda  \neq -1$,  from \eqref{eq:eig-express-2} we have  $B y = \frac{1}{1+\lambda} B S_1^{-1}B^T z$. Using \eqref{eq:eig-express-3}, we have
$\frac{1}{1+\lambda} z= \lambda z.$
Thus,  $z\neq 0$ and  
$\lambda^2+\lambda -1=0,$
that is, $\lambda =\frac{-1\pm \sqrt{5}}{2}$. Since $ z \neq 0 \in \mathbb{R}^{n^2\times 1}$, the eigenvalue  $-1$ has multiplicity $n^2$.
\end{proof}

Finally,  the spectrum of the preconditioned matrix associated with $ \mathcal{M}_3$ is given as follows.

\begin{theorem} 
All of the eigenvalues of $ \mathcal{M}_3^{-1}\mathcal{K}$ are $1$, and the minimal polynomial of this preconditioned matrix is $p(z)=(z-1)^3$.
\label{thm:M3K}
\end{theorem}
\begin{proof}
Using the notation of~\eqref{eq:block_K}, the result follows immediately since $ \mathcal{M}_3^{-1}\mathcal{K}=(LD)^{-1}LDU=U$
\end{proof}

\subsection{Approximations of the Schur complements}\label{sub:schur-approximation}

The  choices $\mathcal{M}_1, \mathcal{M}_2,$ and $\mathcal{M}_3$ as preconditioners  are too computationally costly to work with in practice, so we seek effective approximations. Specifically, in order to make the solver practical, we investigate the structure of the Schur complements $S_1$ and $S_2$, and derive approximations that are easier to compute and invert.

 \subsubsection{Approximations of $S_1$}
 \label{sec:approxS1}
To find good approximations of $S_1$ in~\eqref{eq:S1}, we seek approximations for the action of its additive components, namely $A_s$  and $GA_d^{-1}G^T$. 
Given the sparsity structure of $G^T$, \eqref{eq:G-form}, it follows that $GA_d^{-1}G^T$ is given by
\begin{equation}\label{eq:GAd^{-1}G-approx}
\quad GA_d^{-1}G^T =
 \begin{pmatrix}
 0  &     0   & 0  \\ 
0   &    T      & 0 \\
0 &      0              & 0 
 \end{pmatrix},
\end{equation} 
where $T$ is an $n \times n$ matrix, to be approximated.

Our first (naive) approximation is to take a scaled identity. To that end, we take the diagonal approximation $\left({\rm diag}(A_d)\right) ^{-1}\approx A_d^{-1}$ and ignore the corrections near the boundaries: $T \approx \frac{\tau}{\kappa} I_n$ with $\tau=\frac13$.
The resulting approximation of $S_1$ is
\begin{equation}
\tilde{S}_1 = 
 \begin{pmatrix}
 A_{11}  &    A_{12}  & 0  \\ 
0   &   A_{22}+ \frac{\tau}{\kappa} I_n      &  A_{23} \\
0&      A_{32}              & A_{33} 
 \end{pmatrix}.
\label{eq:tildeS1}
\end{equation}
 In our numerical experiments we have found that this simple approach is effective for a limited range of the  parameters $\kappa$, $\nu$, and $h$. It is thus necessary to consider a more sophisticated alternative, as we do next.

Suppose the Cholesky decomposition of $A_d$ is given by
$$ A_d = F F^T, $$
and let 
$ G A_d^{-1} G^T = W^T W$, where $W=F^{-1}G^T.$ Taking the block structure of $G^T$ into consideration, we partition $F$ as follows: 
\begin{equation*}
F = \begin{pmatrix}
 F_{11} &    0    \\ 
 F_{21} &    F_{22}  
 \end{pmatrix},
 \end{equation*}
where $F_{11} \in {\mathbb R}^{(n^2-n) \times (n^2-n)}$ and $F_{22} \in {\mathbb R}^{n \times n}$.
It readily follows that 
\begin{equation*}
W =\begin{pmatrix}
0 &     0  & 0  \\ 
0   &     F_{22}^{-1} / h   & 0
 \end{pmatrix}
\end{equation*} 
and $$ T=(F_{22}^{-T} F_{22}^{-1}) / h^2, $$
where $F_{22}$ is an $n \times n$  lower triangular matrix.

In practice, since the Cholesky factorization is too expensive to compute, we compute an incomplete Cholesky factorization of $A_d$ with a moderate drop tolerance. We then replace $F_{22}$ by  the corresponding incomplete factor, which we  denote by $\tilde{F}_{22}$. 

Using the above approach, we denote the corresponding approximation to $S_1$ as
\begin{equation}
\hat{S}_1 = 
 \begin{pmatrix}
 A_{11}  &    A_{12}  & 0  \\ 
0   &   A_{22}+  (\tilde{F}_{22}^{-T} \tilde{F}_{22}^{-1}) / h^2    &  A_{23} \\
0&      A_{32}              & A_{33} 
 \end{pmatrix}.
\label{eq:hatS1}
\end{equation}
We have found this approach to be robust with respect to $\kappa$, $\nu$, and $h$; see Section~\ref{sec:numerical}.

 \subsubsection{Approximation of $S_2$}
 \label{sec:approxS2}
Recall from ~\eqref{eq:S2} that $S_2= B S_1^{-1} B^T$. Consider $\tilde{S}_1$ of \eqref{eq:tildeS1}, and let us further sparsify it as follows: we keep the block diagonal part of $\tilde{S}_1$ and $A_{23}$, which contains important information about the interface, and drop the off-diagonal blocks $A_{12}$ and $A_{32}$. We further replace the $(2,2)$ block of the approximation $\tilde{S}_1$ by its diagonal part:
\begin{equation*}
 \widetilde{A}_{22}=\frac{2\nu}{h^2} I_n+\frac{\tau}{\kappa}I_n.
\end{equation*} 
We then use this as a sparser  approximation of $S_1$:
\begin{equation*}
\widecheck{S}_1 = 
 \begin{pmatrix}
 A_{11}  &    0  & 0  \\ 
0   &   \widetilde{A}_{22}      &  A_{23} \\
0&      0            & A_{33} 
 \end{pmatrix}.
 \end{equation*}
 Then we have
\begin{align*}
B \widecheck{S}_1^{-1}B^T & \approx \begin{pmatrix}
B_x &  B_0  & B_y
 \end{pmatrix}
 \begin{pmatrix}
 A_{11}^{-1} &     0   & 0  \\ 
0   &       \widetilde{A}_{22}^{-1}       & -\widetilde{A}_{22}^{-1} A_{23}A_{33}^{-1} \\
0 &      0              & A_{33}^{-1}
 \end{pmatrix} 
\begin{pmatrix}
B_x^T\\
B_0^T\\
B_y^T
 \end{pmatrix} \\
 &=  B_x A_{11}^{-1}B_x^T +B_y A_{33}^{-1}B_y^T+ B_0 \widetilde{A}_{22}^{-1}B_0^T  +B_0 \widetilde{A}_{22}^{-1}A_{23}A_{33}^{-1}B_y^T.
\end{align*}
The matrix $B_x A_{11}^{-1}B_x^T + B_y A_{33}^{-1}B_y^T$ can be approximated by a scaled identity, since in the MAC discretization we have that $B_x B_x^T$ and $B_y B_y^T$ are scaled Laplacians.  In fact,
 \begin{equation*}
B_x A_{11}^{-1}B_x^T +B_y A_{33}^{-1}B_y^T \approx \frac{1}{\nu}I_{n^2-n}.
\end{equation*} 

Then,  
\begin{equation*}
 B_0 \widetilde{A}_{22}^{-1} B_0^T  =\begin{pmatrix} I_n/h  \\ 0 \end{pmatrix} 
\left( \frac{2\nu}{h^2} I_n+\frac{\tau}{\kappa}I_n \right)^{-1} \begin{pmatrix} I_n/h  & 0 \end{pmatrix}= \begin{pmatrix}
 \frac{\kappa}{2 \nu \kappa+h^2 \tau}I_n &     0        \\ 
0     &    0
\end{pmatrix}.
\end{equation*}
 Further, we have
\begin{align*}
B_0 \widetilde{A}_{22}^{-1} A_{23}A_{33}^{-1}B_y^T
 & = \begin{pmatrix} I_n/h  \\ 0 \end{pmatrix} \left( \frac{2\nu}{h^2}  I_n+\frac{\tau}{\kappa}I _n\right)^{-1}  \begin{pmatrix}-\frac{2\nu}{h^2} I_n  &  0 \end{pmatrix} A_{33}^{-1}B_y^T \\
  & = \begin{pmatrix}
 -\frac{2\nu \kappa}{h(2 \nu \kappa+h^2 \tau)} I_n &     0        \\ 
0     &    0
\end{pmatrix} A_{33}^{-1}B_y^T.
 \end{align*} 
This matrix contains entries that are smaller by a factor of $h$ than $ B_0 \widetilde{A}_{22}^{-1} B_0^T$ and therefore we drop it and do not incorporate it into the approximation.

Based on the above, we approximate $S_2$ by
\begin{equation}
\widehat{S}_2 = \frac{1}{\nu} I_{n^2-n} +\begin{pmatrix}
 \frac{\kappa}{2 \nu \kappa+h^2 \tau} I_n &     0        \\ 
0     &    0
\end{pmatrix}=\begin{pmatrix}
 \frac{3\nu \kappa+h^2\tau}{\nu(2 \nu \kappa+h^2 \tau)} I_n &     0        \\ 
0     &     \frac{1}{\nu} I_{n^2-2n} 
\end{pmatrix}.
\label{eq:hatS2}
\end{equation}

\subsubsection{Practical block preconditioners}
Based on the discussion in Subsections \ref{sec:approxS1} and \ref{sec:approxS2},  for our numerical experiments we will consider mostly the following block preconditioners:
\begin{equation*}
 \mathcal{\widehat{M}}_1 =\begin{pmatrix}
 A_d &     0       & 0  \\ 
0    &    -\widehat{S}_1    & 0 \\
0     &     0       & \widehat{S}_2
 \end{pmatrix}, \quad
 \mathcal{\widehat{M}}_2 =\begin{pmatrix}
 A_d &     0       & 0  \\ 
G    &    -\widehat{S}_1    & 0 \\
0     &     0       & \widehat{S}_2
 \end{pmatrix}, \quad 
  \mathcal{\widehat{M}}_3=  \begin{pmatrix}
 A_d &     0       & 0  \\ 
G     &     -\widehat{S}_1    & 0 \\
0     &     B        & \widehat{S}_2
 \end{pmatrix},
 \end{equation*}
where $\widehat{S}_1$ and $\widehat{S}_2$ are given by \eqref{eq:hatS1} and \eqref{eq:hatS2}, respectively.

 \section{Numerical experiments} \label{sec:numerical}
We consider three numerical examples. The first two are taken from \cite{shiue2018convergence}, but with a different formulation of the BJS condition.  We use those examples to perform an error validation and confirm that we observe the expected order of the error. These two examples  impose specific constraints on the values of the physical parameters $\nu,\kappa$. 
 
 We then move to consider a third example from \cite{luo2017uzawa}, where there is no restriction on the physical parameters; this allows us to investigate the convergence behavior of our solver for  a broad range of the parameters. As explained in Section \ref{sec:preconditioner}, we assume Dirichlet boundary conditions in all our examples. Our code is written in {\sc {Matlab}}.

The dimensions of the linear systems used in our numerical experiments are given in Table \ref{tab:dimensions}.
\begin{table} 
 \caption{Values of $n$ and the dimensions of the corresponding linear systems}
\centering
\begin{tabular}{cc}
\hline
$n$     & dimensions          \\ \hline     
32	& 4,064  \\
64	& 16,320  \\
128	& 65,508  \\
256	& 261,888 \\
512	& 1,048,064  \\
1024  & 4,193,280 \\
 \hline
\end{tabular}\label{tab:dimensions}
\end{table}   
 
  {\bf Example 1}:  We take $\Omega_s=[0,1]\times [1,2]$ and $\Omega_d=[0,1]\times [0,1]$. The analytical solution is given by
   \begin{align*}
 u& = -\frac{1}{\pi} e^y\sin(\pi x),\\
 v& =(e^y-e)\cos(\pi x), \\
 p&=2e^y\cos(\pi x),\\
 \phi&=(e^y-ye)\cos(\pi x). 
 \end{align*}
 The interface equations \eqref{eq:interface_conditions} require that $\alpha=\nu=1$. 

 {\bf Example 2}:  We consider $\Omega_s=[0,1]\times [1,2]$ and $\Omega_d=[0,1]\times [0,1]$. The analytical solution is given by
 \begin{align*}
 u& =(y-1)^2+x(y-1)+3x-1,\\
 v& =x(x-1)-0.5(y-1)^2-3y+1, \\
 p&=2x+y-1,\\
 \phi&=x(1-x)(y-1)+\frac{(y-1)^3}{3}+2x+2y+4. 
 \end{align*}
 By \eqref{eq:interface_conditions} it is required that $\alpha=\nu=\kappa=1$.

 {\bf Example 3}:  We consider $\Omega_s=[0,1]\times [0,1]$ and $\Omega_d=[0,1]\times [-1,0]$. The equation is constructed so that the analytical solution is given by
 \begin{align*}
 u& =\eta'(y) \cos x,\\
 v& =\eta(y)\sin x, \\
 p&=0,\\
 \phi&=e^y\sin x,
 \end{align*}
 where 
 \begin{equation*}
\eta(y) = -\kappa -\frac{y}{2\nu}+\left(-\frac{\alpha}{4\nu^2}+\frac{\kappa}{2}\right)y^2.
\end{equation*} 
Using interface condition \eqref{eq:masscon}, there is no constraint on   $\kappa$.
 Using interface condition \eqref{eq:balance}, there is no constraint on   $\nu$.
 Using interface condition \eqref{eq:BJS}, there is no constraint on   $\alpha$ and $\nu$.  Our numerical experiments suggest that $\alpha=\nu$ gives a good performance.

\subsection{Convergence order study}
 
First, we check the convergence order of the velocity and pressure for the three examples.

{\bf Example 1}:  Table \ref{tab:conve_example1} shows the convergence rates for the values of the physical parameters $\alpha=\nu=\kappa=1$.  We observe second-order convergence for the velocity and pressure components for Stokes, while for the Darcy the convergence order of $\phi$ is approximately 1.8.  
\begin{table}[htp]
\caption{Convergence rates for Example 1. Each row shows the ratio between error norms for two adjacent grids.}
\centering
\begin{tabular}{ccccc}
\hline
$n_1/n_2$     & 32/64   &64/128   &128/256    &256/512   \\ \hline     
 $u$    &   1.9888  &  1.9957 &   1.9983 &   1.9994 \\ 
$v$  &  1.9895  &  1.9965  &  1.9990  &  1.9998 \\ 
$p$   & 1.9946  &   1.9982  &   1.9994 &     1.9998 \\
$\phi$    &   1.7136  &  1.7759  &  1.8198  &  1.8514  \\
\hline
\end{tabular}\label{tab:conve_example1}
\end{table}

{\bf Example 2}:  Table \ref{tab:conve_example2} shows the convergence rates for the values of the physical parameters $\alpha=\nu=\kappa=1$. We observe second-order convergence for the pressure components of Stokes and first-order convergence for the remaining components.

\begin{table}[htp]
\caption{Convergence rates for Example 2. Each row shows the ratio between error norms for two adjacent grids.}
\centering
\begin{tabular}{ccccc}
\hline
$n_1/n_2$   &32/64    & 64/128  &128/256   &256/512   \\ \hline     
$u$        &   1.9070  &   1.7649  &  1.4823 &   1.2078 \\
$v$     & 2.0639  &  1.9929   & 1.5441  &  1.0405 \\
$p$    & 2.0035  &  2.0197  &  2.0306  &  2.0009 \\
 $\phi$    & 1.0139  &  1.0072  &  1.0036  &  1.0018  \\
\hline
\end{tabular}\label{tab:conve_example2}
\end{table}

{\bf Example 3}: 
Table \ref{tab:conve_example32} shows the convergence rates for  $\nu=1$ and $\kappa=10^{-2}$, where we observe first-order convergence for all components. This is typical for most values of the physical parameters that we have tested. We note that for $\nu=\kappa=1$ we have observed nearly second-order convergence rates for all components. 

\begin{table}[htp]
\caption{Convergence rates for Example 3 with $\nu=1$ and $\kappa=10^{-2}$. Each row shows the ratio between error norms for two adjacent grids.}
\centering
\begin{tabular}{ccccc}
\hline
$n_1/n_2$   & 32/64  &64/128   &  128/256  & 256/512   \\ \hline   
   $u$     &   1.0386   & 1.0158 &   1.0065 &   1.0027 \\
 $v$    &  1.0940   & 1.0458   & 1.0224   & 1.0110 \\
  $p$ &   1.0767  &  1.0351  &  1.0165   & 1.0079 \\
$\phi$    & 0.9750   & 0.9872   & 0.9935   & 0.9968 \\
\hline
\end{tabular}\label{tab:conve_example32}
\end{table} 

In summary, in all examples we observe either first or second-order convergence, depending on the values of the physical parameters and the model problems. This is in line with or better than the theoretically-guaranteed first-order convergence \cite{shiue2018convergence}. We also note that although the values of the meshsize $h$ used in our tests do not always satisfy \eqref{eq:convcond}, the scheme still converges and we obtain the theoretically-guaranteed first-order convergence. 
 
 In the remainder of this section we conduct our numerical tests using Example 3.

\subsection{Eigenvalue distribution of the double saddle-point matrix (Example 3)}
\begin{figure}[hbt]
\centering
\includegraphics[width=0.49\textwidth]{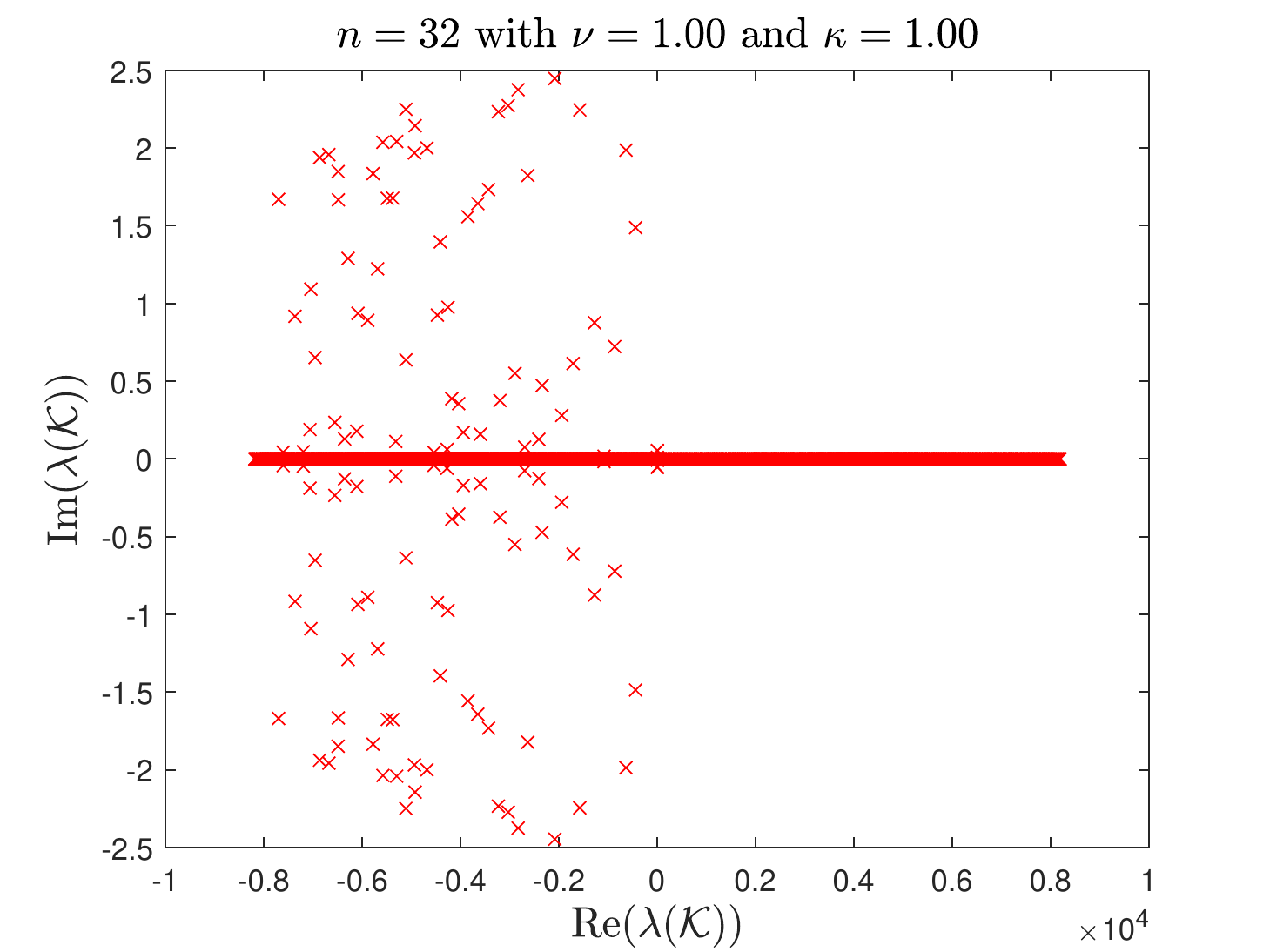}
\includegraphics[width=0.49\textwidth]{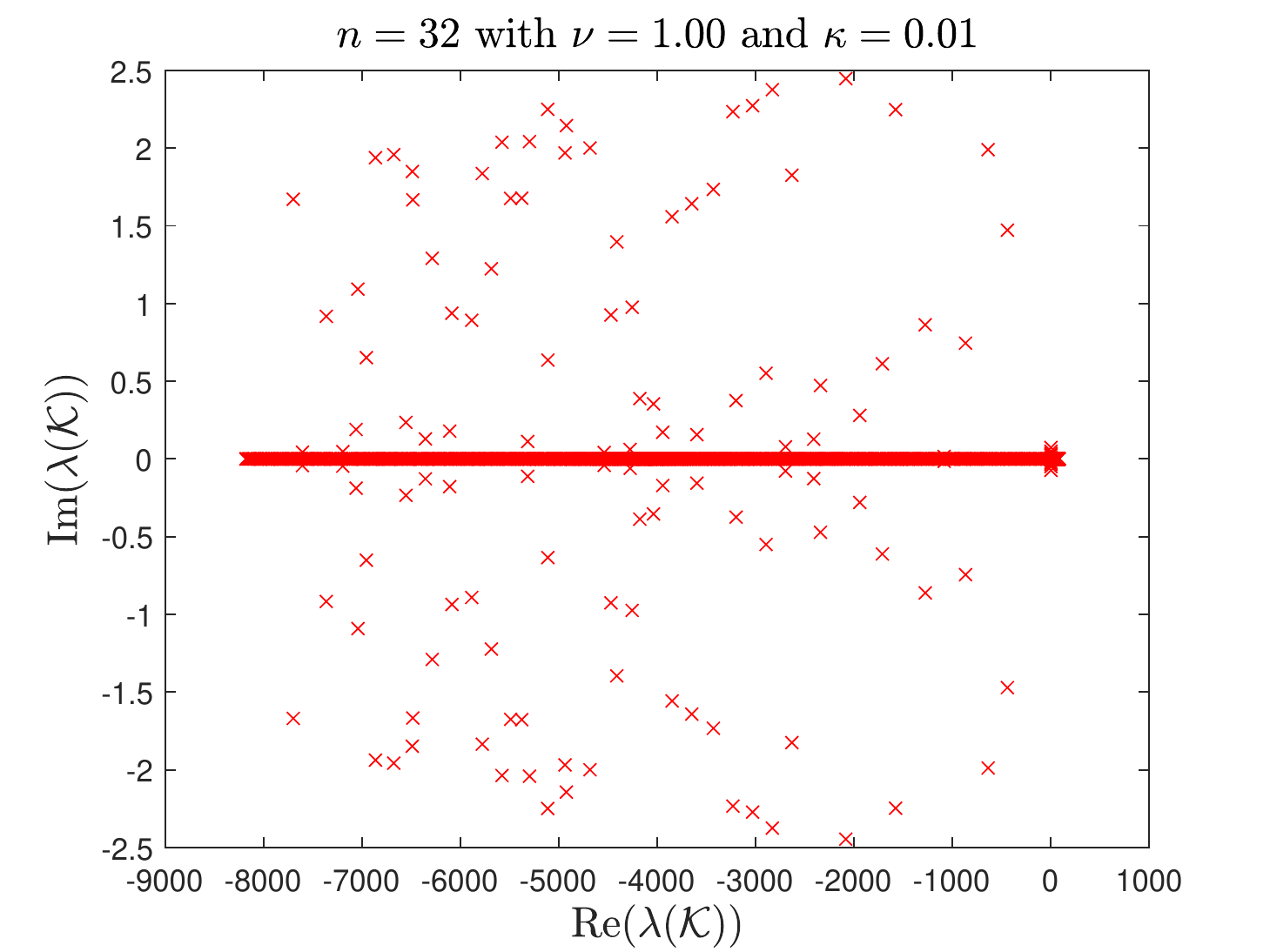}
\includegraphics[width=0.49\textwidth]{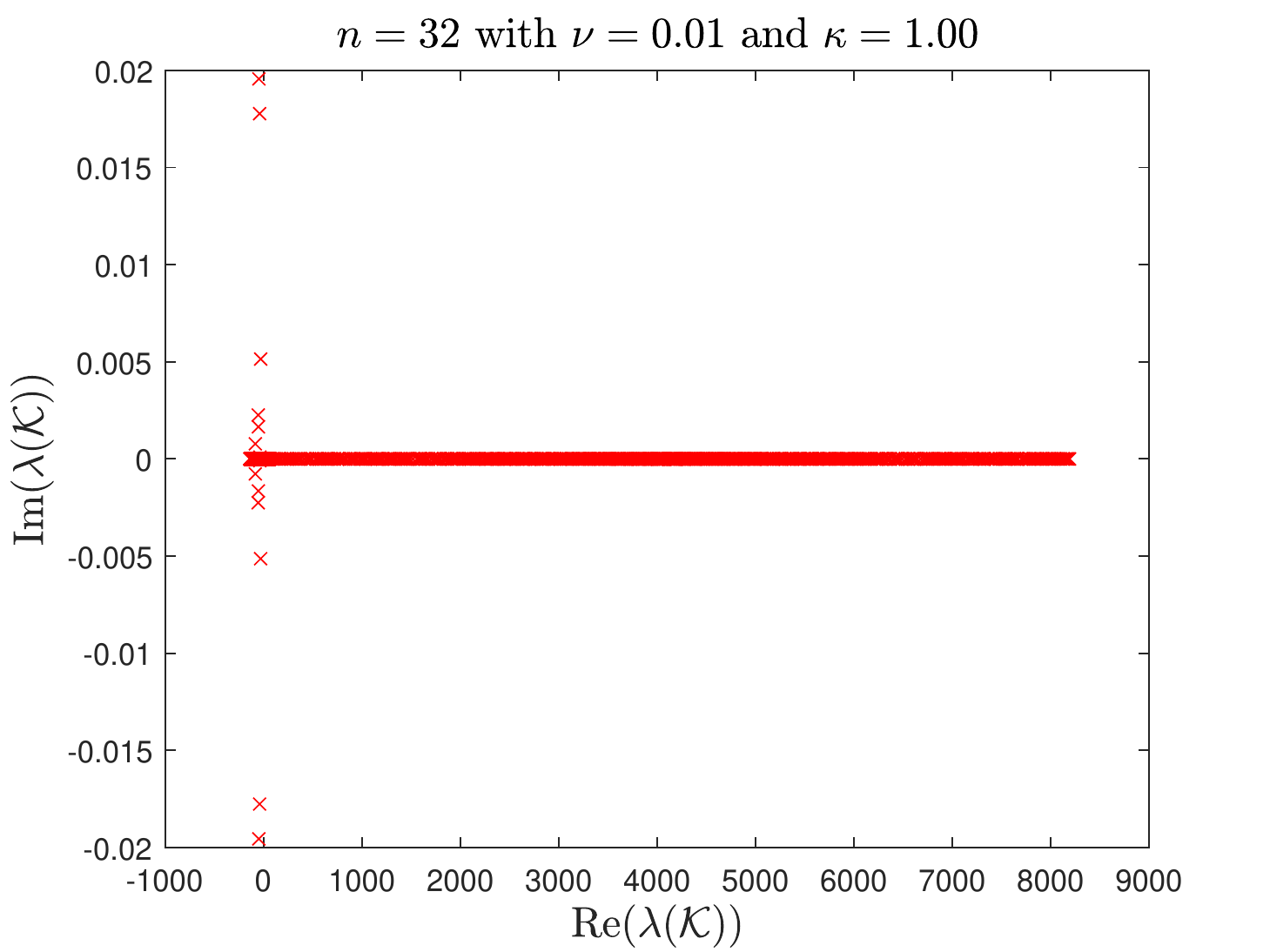}
\includegraphics[width=0.49\textwidth]{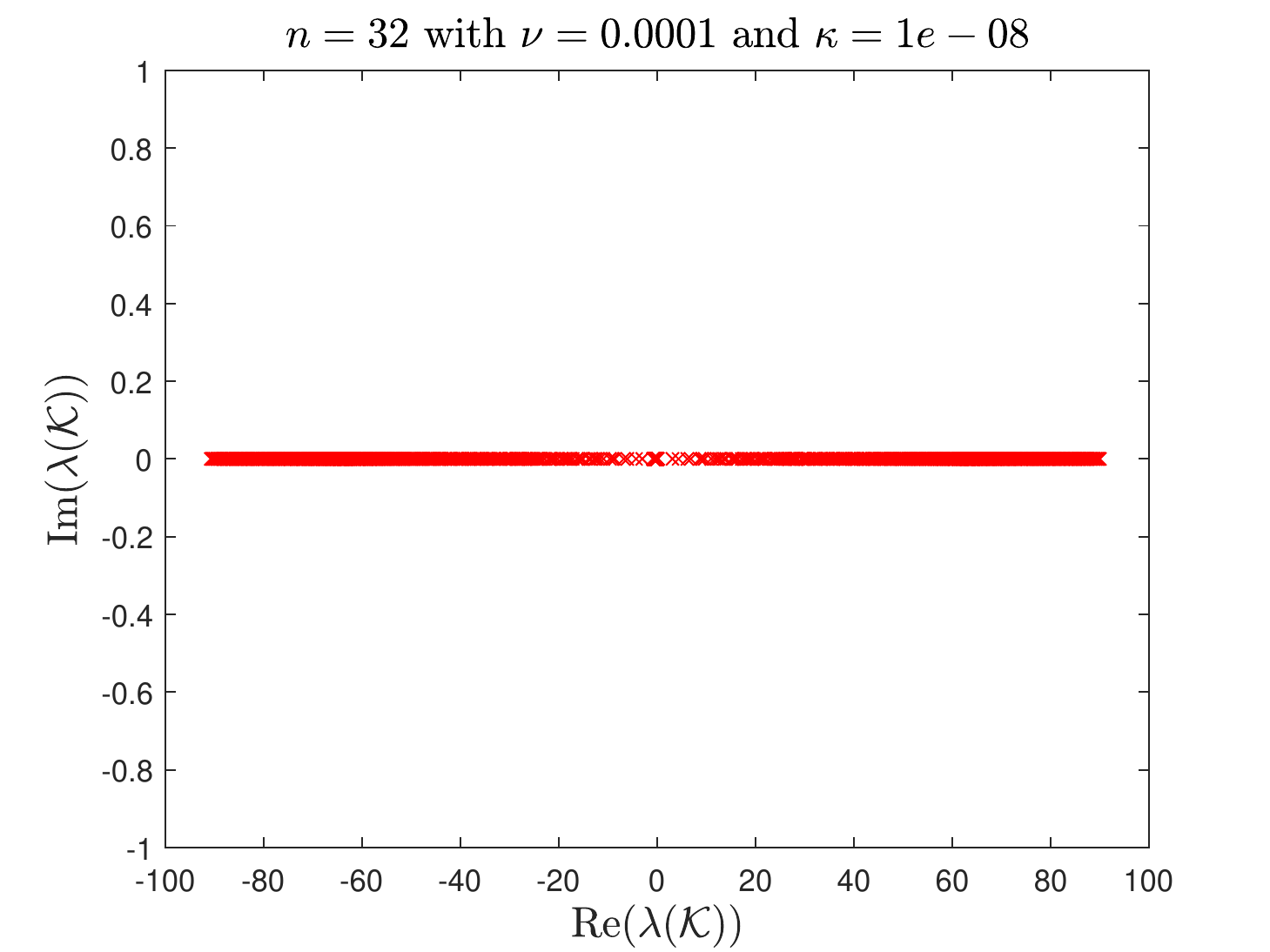}
\caption{The eigenvalue distribution of $\mathcal{K}$ with different values of $\nu$ and $\kappa$. \label{fig:spectrum}}
\end{figure}  

We  explore the effect of $\kappa$ and $\nu$ on the eigenvalue distribution of $\mathcal{K}$ for Example 3. We  take $n=32$ and vary the values of $\kappa$ and $\nu$. The results are shown in Figure~\ref{fig:spectrum}.  Notice that in all examples, the magnitudes of the real parts of the eigenvalues are significantly larger than the magnitudes of the imaginary parts.

We observe that for $\nu=\kappa=1$ (top left plot) the real part of the eigenvalues is spread rather evenly (in terms of magnitudes) over both sides of the real axis. We also notice that the eigenvalues with a negative real part are complex, whereas the eigenvalues on the right half of the plane are real. While the imaginary parts of the eigenvalues do not exceed approximately 2.5, the largest positive and negative real parts are almost $10^4$ in value. 

Taking $\kappa=0.01$ and keeping $\nu=1$ (top right plot) generates a rather dramatic effect on the real part of the eigenvalues; they are shifted towards the negative axis. In our computations we have found that the eigenvalue with the algebraically maximal real part was approximately equal to $81.9$, whereas the eigenvalue with the algebraically minimal real part was approximately $-8,183.0$. 

Taking $\kappa=1$ and $\nu=0.01$ (bottom left plot) shifts the real parts of the eigenvalues to be mostly positive. The scales of the imaginary parts are now smaller. The algebraically smallest eigenvalue in this case was $-0.4$ and the algebraically largest eigenvalue was approximately $8,189.5$.

Finally, we show the interesting case where $\nu=10^{-4}$ and $\kappa=10^{-8}$ (bottom right plot). All eigenvalues in this case are real and are spread over both axes in a rather symmetrical fashion. The algebraically maximal value in this case was $90.0$ and the  algebraically minimal one was $-90.8$.

The above observations indicate that the spectral properties of the coefficient matrix highly depend on the values of the physical parameters $\kappa$ and $\nu$.

\subsection{GMRES performance}

In our numerical tests we run GMRES(20) and stop the iteration once the initial relative residual is reduced by a factor of $10^{-8}$ or a maximum iteration count of 500 iterations has been reached. For the incomplete Cholesky factorization of the Schur complement $S_1$, we use a drop tolerance of $10^{-2}$.

In Table \ref{tab:M1M2-inexact} we report the iteration counts of preconditioned GMRES using preconditioners $\widehat{\mathcal{M}}_1$ and $\widehat{\mathcal{M}}_2$. We see that these two preconditioners scale poorly with respect to small physical parameters. To better understand this behavior, we explore an improved version of the preconditioner, where we use the approximation $\widehat{S}_1$ and exact $S_2$ for the Schur complements in $\mathcal{M}_1$ and $\mathcal{M}_2$. We report the corresponding results in Table \ref{tab:M1M2-inexactS1-exactS2}. We see a much better performance. However, the cost of inverting $S_2$ exactly is too high in practice, and we seek less costly alternatives. We thus consider approximations of $\mathcal{M}_3$: we use the simple  approximations $\widehat{S}_1$ and  $\widehat{S}_2$ defined in \eqref{eq:hatS1} and \eqref{eq:hatS2}, respectively, and include the block $B$  to couple the Stokes velocity and Darcy variable. This is the preconditioning approach that we have found to be the most promising.

 \begin{table}[htp]
 \caption{Iteration counts of GMRES(20) for the preconditioners  $\widehat{\mathcal{M}}_1$ and  $\widehat{\mathcal{M}}_2$  with $\nu=1$ and varying $n$ and $\kappa$.
The symbol `-' marks no convergence to a relative residual tolerance of $10^{-8}$ within 500 iterations. The two schemes failed to converge for $\kappa<10^{-4}$.}
\centering
\begin{tabular}{l|ccc|ccc }
\hline
  \multirow{1}{*}{\diagbox{$\kappa$}{$n$} } &
      \multicolumn{3}{c}{$\widehat{\mathcal{M}}_1$} &
      \multicolumn{3}{c}{$\widehat{\mathcal{M}}_2$} \\      
   &$32$      &$64$     &$128$    
   &$32$      &$64$     &$128$   
 \\ \hline
$10^{0}  $  	& 60   &  62      	&      60      	&   55   	&  57  	&    62   	\\
$10^{-1} $ 	&67    &  75      	&      87      	&   62   	&  64 		&    70     	\\
$10^{-2} $ 	&186  &   215     	&    275     	&   67   	&  125 	&    114    	\\
$10^{-3} $ 	&-     	&   -     	&        -    	&   99		&  159 	&    204     	\\
$10^{-4} $	& 444 &   285     	&    -     	&   239   	&   78 	&     -         	\\
$10^{-5} $ 	& -  	&   -     	&       -     	&    -  	&   - 		&      -     	\\
\hline
 \end{tabular}\label{tab:M1M2-inexact}
\end{table}

 \begin{table}[htp]
 \caption{Iteration counts of GMRES(20) for the inexact versions $\mathcal{M}_{1,in}$ and $\mathcal{M}_{2,in}$ corresponding to preconditioners  $\mathcal{M}_1$ and  $\mathcal{M}_2$  with $\nu=1$ and varying $n$ and $\kappa$, using approximation $\widehat{S}_1$ and  the exact $S_2$.}
\centering
\begin{tabular}{l|cc|cc }
\hline
  \multirow{1}{*}{\diagbox{$\kappa$}{$n$} } &
      \multicolumn{2}{c}{$\mathcal{M}_{1,in}$} &
      \multicolumn{2}{c}{$\mathcal{M}_{2,in}$} \\      
   &$32$      &$64$       
   &$32$      &$64$       
 \\ \hline
$10^{0}  $  	&        14 	&   15      	&    10	&   	  11    	\\
$10^{-1} $ 	&         17	&    19     	&    12	&   	   14   	\\
$10^{-2} $ 	&        25	&    26      	&    15	&   	   16   	\\
$10^{-3} $ 	&        33 	&    35      	&    17	&   	   21   	\\
$10^{-4} $	&        34 	&   40       	&    17	&   	   21   	\\
$10^{-5} $ 	&        29 	&   38       	&    16	&   	    21  	\\
$10^{-6} $	&        24 	&   34       	&    15	&   	   19   	\\
$10^{-7} $ 	&        25 	&   31       	&    15	&   	    17  	\\
$10^{-8} $ 	&        22 	&   31       	&    14	&   	    18  	\\
\hline
 \end{tabular}\label{tab:M1M2-inexactS1-exactS2}
\end{table}

As per Theorem \ref{thm:M3K}, the preconditioned matrix $\mathcal{M}_3^{-1} \mathcal{K}$ has one eigenvalue 1 with a minimal polynomial of degree 3.  We have confirmed for this ideal (yet impractical) preconditioner that GMRES takes three iterations to converge. 

In all our experiments reported below, we use the approximation $\hat{S}_2$ in \eqref{eq:hatS2} for $S_2$; we have found this approximation to be robust with respect to the physical parameters. On the other hand, the quality of the approximation of $S_1$ has a more dramatic effect on convergence of GMRES, as we discuss below.

In Table \ref{tab:M3simple} we show that the approximation of $S_1$ based on the scaled identity approximation of $T$, namely $\tilde{S}_1$ given in \eqref{eq:tildeS1}, is only effective for relatively large values of $\nu$ and $\kappa$. We set $\nu=1$ and  observe a good degree of scalability (nearly constant iteration counts) for $\kappa=1$ and $\kappa=0.1$, but convergence starts degrading for smaller values of $\kappa$, with poor convergence for $\kappa \leq 10^{-4}$.

\begin{table} 
 \caption{Iteration counts of GMRES(20) with an inexact version of $\mathcal{M}_3$, using a scaled identity approximation of $S_1$ and $\widehat{S}_2$ with $\nu=1$ and varying $n$ and $\kappa$. The symbol `-' marks no convergence to a relative residual tolerance of $10^{-8}$ within 500 iterations.}
\centering
\begin{tabular}{c|ccccccccc}
\hline
\diagbox{$n$}{$\kappa$}     &  $10^0$ & $10^{-1}$ & $10^{-2}$ & $10^{-3}$  & $10^{-4}$ & $10^{-5}$ & $10^{-6}$  & $10^{-7}$ & $10^{-8}$     \\ \hline     
32 &     18  & 19  & 21   & 37  &  49   & 76  & 79  & 360   & -  \\
64 &     18  & 19  & 24   & 39  &  75   & -  & -  & -   & -  \\
128 &   19  & 20  & 25   & 44  & 280    &  - & -  & -   & -  \\
256 &   20  & 21  & 28   & 44  & -    & -  & -  & -   & -  \\
512 &  21   &  22 &   31  &  39  &  448 &  105 &  -  & -&  - \\
1024 & 22  & 23   &  31     &  37     &  464     &  300     &  -  & -  & - \\
\hline
\end{tabular}\label{tab:M3simple}
\end{table} 

In Tables \ref{tab:nu1e0inexact} and \ref{tab:nu1em2inexact} we consider the much superior approximation of $S_1$ based on the incomplete Cholesky factorization with drop tolerance $10^{-2}$, namely $\hat{S}_1$ defined in \eqref{eq:hatS1}. We see that for both values of $\nu$ and varying values of $\kappa$, the preconditioner $\widehat{\mathcal{M}}_3$  is quite robust, although convergence degrades as $\kappa$ becomes smaller. In Table \ref{tab:nu1em2exactS1} we replace the approximation of $\hat{S}_1$ by the exact $S_1$, just to confirm that indeed, the source of the decline in performance for small values of $\kappa$ is related to the quality of the approximation of $S_1$. We therefore expect that a better approximation, for example an incomplete Cholesky factorization with a tighter drop tolerance would yield faster convergence in most cases. 

\begin{table}
 \caption{Iteration counts of GMRES(20) for  the preconditioner $\widehat{\mathcal{M}}_3$ with $\nu=1$ and varying $n$ and $\kappa$.}
\centering
\begin{tabular}{c|ccccccccc}
\hline
\diagbox{$n$}{$\kappa$}     &  $10^0$ & $10^{-1}$ & $10^{-2}$ & $10^{-3}$  & $10^{-4}$ & $10^{-5}$ & $10^{-6}$  & $10^{-7}$ & $10^{-8}$     \\ \hline     
32 &     18  &  17  &  18  &  18 &   18  &  18  &  20 &   21  &  23 \\
64 &     19  &  19  &  19  &  20 &   21  &  23  &  24 &   38  &  39 \\
128 &   20  &  20  &  20  &  23 &   24  &  35  &  37 &   37  &  38 \\
256 &   21  &  22  &  22  &  25 &   37  &  32  &  35 &   37  &  39 \\
512 &   22  &  23  &  23  &  36 &   36  &  34  &  38 &   39  &  42 \\
1024 & 24  &  25  &  24  &  39 &   37  &  41  &  59 &   60  &  61 \\

 \hline
\end{tabular}\label{tab:nu1e0inexact}
\end{table}

\begin{table} 
 \caption{Iteration counts of GMRES(20) for  the preconditioner $\widehat{\mathcal{M}}_3$ with $\nu=10^{-2}$ and varying $n$ and $\kappa$.}
\centering
\begin{tabular}{c|ccccccccc}
\hline
\diagbox{$n$}{$\kappa$}     &  $10^0$ & $10^{-1}$ & $10^{-2}$ & $10^{-3}$  & $10^{-4}$ & $10^{-5}$ & $10^{-6}$  & $10^{-7}$ & $10^{-8}$     \\ \hline     
32  &   16   & 15  & 16   & 16  & 17    & 19  & 20  & 37    &  39  \\
64 &    17   &  16 & 17   & 18  & 20    &  21  & 35  & 36   & 38  \\
128 &  18   & 18  & 18   &  11 &  21   & 32  &  33  & 35   &  37  \\
256 &  18    & 20   & 21   & 11  & 11    & 11  & 11  & 11   & 11  \\
512 &  20   & 30  & 14   & 13   & 12    & 12  & 11  & 11   & 11  \\
1024 & 20  & 32  & 16   & 14  & 13    & 13  & 12  & 12   & 12  \\

 \hline
\end{tabular}\label{tab:nu1em2inexact}
\end{table}

\begin{table}
 \caption{Iteration counts of GMRES(20) for  the inexact version of preconditioner $\mathcal{M}_3$ with  $\nu=10^{-2}$ and varying $n$ and $\kappa$, using  the exact $S_1$ and approximation $\widehat{S}_2$.}
\centering
\begin{tabular}{c|ccccccccc}
\hline
\diagbox{$n$}{$\kappa$}     &  $10^0$ & $10^{-1}$ & $10^{-2}$ & $10^{-3}$  & $10^{-4}$ & $10^{-5}$ & $10^{-6}$  & $10^{-7}$ & $10^{-8}$     \\ \hline     
32 &    14   &  14   &  15   &  15  & 16    & 17   & 19   & 20    & 22  \\
64 &     14  &  14  &  15  & 15  & 15    &  17 & 19  & 20   & 31  \\
128 &    14 &  14  & 14   & 7  & 15    & 16  & 18  & 20   & 37  \\
 \hline
\end{tabular}\label{tab:nu1em2exactS1}
\end{table}  
   
Finally, in Table~\ref{tab:nu1em4inexact} we show that when the difference in scale between $\nu$ and $\kappa$ is smaller, then preconditioned GMRES with $\widehat{\mathcal{M}}_3$ performs remarkably well even when the parameters are small. 

\begin{table}
 \caption{Iteration counts of GMRES(20) for the preconditioner $\widehat{\mathcal{M}}_3$ with $\nu=10^{-4}$ and varying $n$ and $\kappa$.}
\centering
\begin{tabular}{c|ccccccccc}
\hline
\diagbox{$n$}{$\kappa$}     &  $10^0$ & $10^{-1}$ & $10^{-2}$ & $10^{-3}$  & $10^{-4}$ & $10^{-5}$ & $10^{-6}$  & $10^{-7}$ & $10^{-8}$     \\ \hline     
32  &  9   & 8  &  7  & 7  & 7    & 7  & 7  & 7   & 7  \\
64  &  9  & 8  &  6  & 6  & 6    & 6  & 6  & 6   & 6  \\
128  &  10   & 7  &  6  & 6  & 6    & 6  & 6  & 6   & 6  \\
256  &  11   & 8  &  6  & 6  & 6    & 6  & 6  & 6   & 6  \\
512  &  12   & 9  &  7  & 6  & 6    & 6  & 6  & 6   & 6  \\
1024  &  14   & 9  &  7  & 6  & 5    & 5  &  5 & 5   & 5  \\
 \hline
\end{tabular}\label{tab:nu1em4inexact}
\end{table}

 \section{Concluding remarks}\label{sec:conclusion}
 We have considered the MAC discretization of the Stokes--Darcy equations and have designed a robust and scalable preconditioner for the corresponding linear system. Our conclusions are: (i) The MAC discretization gives rise to attractive sparsity patterns of some of the block matrices, which we are able to take advantage of for approximating the Schur complements. (ii) It is crucial to include the coupling equations (interface conditions) in the preconditioner. (iii) The nonsymmetry of the coefficient matrix is mild and it is possible to design a solver based on spectral considerations. The analysis reveals a rich and interesting spectral structure. 
 
 The inexact block lower triangular preconditioner $\widehat{\mathcal{M}}_3$ seems promising in terms of robustness with respect to the values of the   physical parameters. Among its attractive features is our ability to form effective and relatively cheap approximations of the Schur complements $S_1$ and $S_2$.

 
 \appendix
 
 \section{Related block preconditioners} 
 
We have considered several additional options for block preconditioners, with some minor changes (e.g., sign changes) in comparison to the ones we have analyzed in Section~\ref{sec:spectral}:
\begin{equation*}
 \mathcal{\tilde{M}}_1 =\begin{pmatrix}
 A_d &     0       & 0  \\ 
0    &    -S_1    & 0 \\
0     &     0       & S_2
 \end{pmatrix}, \quad
 \mathcal{\tilde{M}}_2 =\begin{pmatrix}
 A_d &     0       & 0  \\ 
G    &    -S_1    & 0 \\
0     &     0       & S_2
 \end{pmatrix},\quad 
  \mathcal{\tilde{M}}_3=  \begin{pmatrix}
 A_d &     0       & 0  \\ 
G     &     S_1    & 0 \\
0     &     B        & S_2
 \end{pmatrix}. 
 \end{equation*}
The preconditioned matrix  $\mathcal{\tilde{M}}_1^{-1} \mathcal{K}$ has a large number of complex eigenvalues. The preconditioned matrix $\mathcal{\tilde{M}}_2^{-1} \mathcal{K}$ has three distinct eigenvalues: the eigenvalue 1 with algebraic multiplicity $2n^2-n$ and the complex eigenvalues $\frac{1 \pm \sqrt{3} \imath}{2}$ ($\imath^2=-1$) with multiplicity $n^2$ each. Compare this with 
$\mathcal{M}_2^{-1} \mathcal{K}$, which has four distinct eigenvalues, as per Theorem~\ref{thm:M2K}. 
 The preconditioned matrix $\mathcal{\tilde{M}}_3^{-1} \mathcal{K}$ has three distinct eigenvalues: the eigenvalue 1 with algebraic multiplicity $n^2$, the eigenvalue $-1$ with algebraic multiplicity $n^2-n$, and the eigenvalues $\pm \sqrt{2}-1$ with multiplicities $n^2$ each. We prove these results below.

\begin{theorem} 
The eigenvalues of $ \mathcal{\tilde{M}}_2^{-1}\mathcal{K}$ are 
\begin{enumerate}[(i)]
\item $1$ with multiplicity $2n^2-n$; 
\item $\frac{1\pm \sqrt{3} i}{2}$ with multiplicity $n^2$ each.
\end{enumerate}
\end{theorem}
\begin{proof}
The preconditioned matrix is given by
 \begin{equation*} 
 \mathcal{\tilde{M}}_2^{-1}\mathcal{K} =\begin{pmatrix}
 I &     A_d^{-1} G^T      & 0  \\ 
0    &    I    & -S_1^{-1}B^T \\
0     &    S_2^{-1} B      &0
 \end{pmatrix}.
\end{equation*}
Let $\begin{pmatrix} x^T &  y^T & z^T \end{pmatrix}^T$ be an eigenvector of $ \mathcal{\tilde{M}}_2^{-1}\mathcal{K}$ associated with eigenvalue $\lambda$.
We write the corresponding eigenvalue problem as follows:
\begin{subequations}
\label{eq:eig}
\begin{align}
x +A_d^{-1}G^Ty &=  \lambda x, \label{eq:eig-express-1-hatM2}\\
y-S_1^{-1}B^Tz&=\lambda y, \label{eq:eig-express-2-hatM2}\\
 ( B S_1^{-1}B^T)^{-1} B y&= \lambda z. \label{eq:eig-express-3-hatM2}
\end{align}
\end{subequations}
We have $\begin{pmatrix} x^T &  y^T & z^T \end{pmatrix}^T=\begin{pmatrix} x^T  &  0  & 0 \end{pmatrix}^T$ where $x \neq 0$ is an eigenvector of $ \mathcal{\tilde{M}}_2^{-1}\mathcal{K}$ with  $\lambda =1$.  Since $x\in \mathbb{R}^{n^2\times 1}$,  $1$ is an eigenvalue with multiplicity $n^2$.

If $\lambda=1$ and $y \neq 0$, 
the three equations of \eqref{eq:eig} are simplified to
\begin{subequations}
\begin{align}
A_d^{-1}G^Ty &= 0, \label{eq:eig-express-11-hatM2}\\
 B^Tz&=0,  \label{eq:eig-express-12-hatM2}\\
 B y&=0. \label{eq:eig-express-13-hatM2}
\end{align}
\end{subequations}
Since $B^T$ has full rank,  \eqref{eq:eig-express-12-hatM2} leads to $z=0$.
From \eqref{eq:eig-express-13-hatM2} we have $By=0$. Since  $B \in \mathbb{R}^{n^2 \times (2n^2-n)}$ has rank $n^2$, the  null space of  $B$  has dimension $(2n^2-n)-n^2=n^2-n$. From the proof of Theorem \ref{thm:eigs-M1K}, $y$ satisfies $G^Ty=0$. Thus, the multiplicity of $1$ with eigenvector $\begin{pmatrix} x^T  &  y^T  & 0 \end{pmatrix}^T$ with $y\neq 0$ is $n^2-n$. Therefore, $1$ has multiplicity $2n^2-n$.

If $\lambda  \neq 1$,  from \eqref{eq:eig-express-2-hatM2}, we have  $B y = \frac{1}{1-\lambda} B S_1^{-1}B^T z$. Using \eqref{eq:eig-express-3-hatM2}, we have
\begin{equation*}
\frac{1}{1-\lambda} z= \lambda z.
\end{equation*}
Thus,  $z\neq 0$ and  
\begin{equation*}
\lambda^2-\lambda +1=0,
\end{equation*} 
that is $\lambda =\frac{1\pm \sqrt{3}i}{2}$. Since $ z \neq 0 \in \mathbb{R}^{n^2\times 1}$, the eigenvalues  $\frac{1\pm \sqrt{3}i}{2}$ have multiplicity $n^2$  each.
\end{proof}

\begin{theorem} 
The eigenvalues of $ \mathcal{\tilde{M}}_3^{-1}\mathcal{K}$ are 
\begin{enumerate}[(i)]
\item $1$ with multiplicity $n^2$; 
\item $-1$ with multiplicity $n^2-n$;  
\item  $\sqrt{2}-1\approx 0.4142 $ and $-\sqrt{2}-1 \approx  -2.4142$  with multiplicity   $n^2$ each.
\end{enumerate}
\end{theorem}
 
 \begin{proof}
The preconditioned matrix is given by
 \begin{equation*} 
 \mathcal{\tilde{M}}_3^{-1}\mathcal{K} =\begin{pmatrix}
 I &     A_d^{-1} G^T      & 0  \\ 
0    &    -I    & S_1^{-1}B^T \\
0       &2S_2^{-1} S_1^{-1}         &-I
 \end{pmatrix}.
\end{equation*}
Thus, $n^2$ of the eigenvalues of $\mathcal{\tilde{M}}_3^{-1}\mathcal{K}$ are $1$, and the remaining ones are the eigenvalues of 
\begin{equation*}
H = \begin{pmatrix}
   -I    & S_1^{-1}B^T  \\
 2S_2^{-1} S_1^{-1}       &-I
 \end{pmatrix}.
\end{equation*} 
We write the corresponding eigenvalue problem for $H$ and obtain
\begin{subequations}
\begin{align}
-y +S_1^{-1}B^Tz &=\lambda y, \label{eq:hatM3K-yz1}\\
2 S_2^{-1}By-z &=\lambda z. \label{eq:hatM3K-yz2}
\end{align} 
\end{subequations}
If $\lambda =-1$, then
\begin{align*}
 S_1^{-1}B^Tz &=0\\
2 S_2^{-1}By &=0
\end{align*} 
Therefore, $B^Tz=0$ and $By=0$. Since $B$ is full rank, $z=0$ and   $y$ is the null space of $B$ with dimension $(2n^2-n)-n^2=n^2-n$.

If $\lambda \neq -1$, from \eqref{eq:hatM3K-yz1} we have $y=(1+\lambda)^{-1}S_1^{-1}B^Tz$. Therefore $y, z\neq 0$.  From \eqref{eq:hatM3K-yz2}, we have
\begin{equation*}
(1+\lambda) z= 2S_2^{-1} By=2S_2^{-1} (1+\lambda)^{-1}S_1^{-1}B^Tz=2(1+\lambda)^{-1} z,
\end{equation*} 
which gives
$(1+\lambda)^2=2.$
Therefore $\lambda =\pm \sqrt{2}-1$. Since $B^T$ has  full rank,  the eigenvalues $\pm \sqrt{2}-1$ have multiplicity $n^2$  each.
 \end{proof}

\bibliographystyle{siam}
\bibliography{ref_SD}

\begin{thebibliography}{10}

\bibitem{beik18}
{\sc F.~Ali~Beik and M.~Benzi}, {\em Iterative methods for double saddle point
  systems}, SIAM Journal on Matrix Analysis and Applications, 39 (2018),
  pp.~902--921.

\bibitem{babuvska2010residual}
{\sc I.~Babu{\v{s}}ka and G.~N. Gatica}, {\em A residual-based a posteriori
  error estimator for the {S}tokes--{D}arcy coupled problem}, SIAM Journal on
  Numerical Analysis, 48 (2010), pp.~498--523.

\bibitem{beigl20}
{\sc A.~Beigl, J.~Sogn, and W.~Zulehner}, {\em Robust preconditioners for
  multiple saddle point problems and applications to optimal control problems},
  SIAM Journal on Matrix Analysis and Applications, 41 (2020), pp.~1590--1615.

\bibitem{beik2022preconditioning}
{\sc F.~P.~A. Beik and M.~Benzi}, {\em Preconditioning techniques for the
  coupled {S}tokes--{D}arcy problem: spectral and field-of-values analysis},
  Numerische Mathematik, 150 (2022), pp.~257--298.

\bibitem{bernardi2008mortar}
{\sc C.~Bernardi, T.~C. Rebollo, F.~Hecht, and Z.~Mghazli}, {\em Mortar finite
  element discretization of a model coupling {D}arcy and {S}tokes equations},
  ESAIM: Mathematical Modelling and Numerical Analysis, 42 (2008),
  pp.~375--410.

\bibitem{bradley2021eigenvalue}
{\sc S.~Bradley and C.~Greif}, {\em Eigenvalue bounds for double saddle-point
  systems}, IMA Journal on Numerical Analysis,  (2022).

\bibitem{cai21}
{\sc M.~Cai, G.~Ju, and J.~Li}, {\em Schur complement based preconditioners for
  twofold and block tridiagonal saddle point problems},  (2021).
\newblock https://arxiv.org/abs/2108.08332.

\bibitem{cai2009preconditioning}
{\sc M.~Cai, M.~Mu, and J.~Xu}, {\em Preconditioning techniques for a mixed
  {S}tokes/{D}arcy model in porous media applications}, Journal of
  Computational and Applied Mathematics, 233 (2009), pp.~346--355.

\bibitem{caiazzo2014classical}
{\sc A.~Caiazzo, V.~John, and U.~Wilbrandt}, {\em On classical iterative
  subdomain methods for the {S}tokes--{D}arcy problem}, Computational
  Geosciences, 18 (2014), pp.~711--728.

\bibitem{cao2010coupled}
{\sc Y.~Cao, M.~Gunzburger, F.~Hua, and X.~Wang}, {\em Coupled {S}tokes-{D}arcy
  model with beavers-joseph interface boundary condition}, Communications in
  Mathematical Sciences, 8 (2010), pp.~1--25.

\bibitem{chen2016finite}
{\sc L.~Chen}, {\em Finite difference method for {S}tokes equations: {MAC}
  scheme},  (2016).

\bibitem{chen2016weak}
{\sc W.~Chen, F.~Wang, and Y.~Wang}, {\em Weak {G}alerkin method for the
  coupled {D}arcy--{S}tokes flow}, IMA Journal of Numerical Analysis, 36
  (2016), pp.~897--921.

\bibitem{chidyagwai2016constraint}
{\sc P.~Chidyagwai, S.~Ladenheim, and D.~B. Szyld}, {\em Constraint
  preconditioning for the coupled {S}tokes--{D}arcy system}, SIAM Journal on
  Scientific Computing, 38 (2016), pp.~A668--A690.

\bibitem{discacciati2009navier}
{\sc M.~Discacciati, A.~Quarteroni, et~al.}, {\em Navier-{S}tokes/{D}arcy
  coupling: modeling, analysis, and numerical approximation}, Rev. Mat.
  Complut, 22 (2009), pp.~315--426.

\bibitem{discacciati2007robin}
{\sc M.~Discacciati, A.~Quarteroni, and A.~Valli}, {\em Robin--robin domain
  decomposition methods for the {S}tokes--{D}arcy coupling}, SIAM Journal on
  Numerical Analysis, 45 (2007), pp.~1246--1268.

\bibitem{duretz2011discretization}
{\sc T.~Duretz, D.~A. May, T.~Gerya, and P.~Tackley}, {\em Discretization
  errors and free surface stabilization in the finite difference and
  marker-in-cell method for applied geodynamics: A numerical study},
  Geochemistry, Geophysics, Geosystems, 12 (2011).

\bibitem{eymard2010convergence}
{\sc R.~Eymard, T.~Gallou{\"e}t, R.~Herbin, and J.-C. Latch{\'e}}, {\em
  Convergence of the {MAC} scheme for the compressible {S}tokes equations},
  SIAM Journal on Numerical Analysis, 48 (2010), pp.~2218--2246.

\bibitem{fu2018strongly}
{\sc G.~Fu and C.~Lehrenfeld}, {\em A strongly conservative hybrid {DG}/mixed
  {FEM} for the coupling of {S}tokes and {D}arcy flow}, Journal of Scientific
  Computing, 77 (2018), pp.~1605--1620.

\bibitem{GL1996}
{\sc V.~Girault and H.~Lopez}, {\em Finite-element error estimates for the
  {MAC} scheme}, IMA J. Numer. Anal., 16 (1996), pp.~347--379.

\bibitem{HW1998}
{\sc H.~Han and X.~Wu}, {\em A new mixed finite element formulation and the
  {MAC} method for the {S}tokes equations}, SIAM J. Numer. Anal., 35 (1998),
  pp.~560--571.

\bibitem{harlow1965numerical}
{\sc F.~H. Harlow and J.~E. Welch}, {\em Numerical calculation of
  time-dependent viscous incompressible flow of fluid with free surface}, The
  physics of fluids, 8 (1965), pp.~2182--2189.

\bibitem{hessari2015pseudospectral}
{\sc P.~Hessari}, {\em Pseudospectral least squares method for
  {S}tokes--{D}arcy equations}, SIAM Journal on Numerical Analysis, 53 (2015),
  pp.~1195--1213.

\bibitem{holter2021robust}
{\sc K.~E. Holter, M.~Kuchta, and K.-A. Mardal}, {\em Robust preconditioning
  for coupled {S}tokes--{D}arcy problems with the {D}arcy problem in primal
  form}, Computers $\&$ Mathematics with Applications, 91 (2021), pp.~53--66.

\bibitem{hou2019solution}
{\sc Y.~Hou and Y.~Qin}, {\em On the solution of coupled {S}tokes/{D}arcy model
  with {B}eavers--{J}oseph interface condition}, Computers $\&$ Mathematics
  with Applications, 77 (2019), pp.~50--65.

\bibitem{hm19}
{\sc N.~Huang and C.-F. Ma}, {\em Spectral analysis of the preconditioned
  system for the 3 \texttimes{} 3 block saddle point problem}, Numer.
  Algorithms, 81 (2019), p.~421–444.

\bibitem{karper2009unified}
{\sc T.~Karper, K.-A. Mardal, and R.~Winther}, {\em Unified finite element
  discretizations of coupled {D}arcy--{S}tokes flow}, Numerical Methods for
  Partial Differential Equations: An International Journal, 25 (2009),
  pp.~311--326.

\bibitem{keyes2013multiphysics}
{\sc D.~E. Keyes, L.~C. McInnes, C.~Woodward, W.~Gropp, E.~Myra, M.~Pernice,
  J.~Bell, J.~Brown, A.~Clo, J.~Connors, et~al.}, {\em Multiphysics
  simulations: {C}hallenges and opportunities}, The International Journal of
  High Performance Computing Applications, 27 (2013), pp.~4--83.

\bibitem{lai2019simple}
{\sc M.-C. Lai, M.-C. Shiue, and K.~C. Ong}, {\em A simple projection method
  for the coupled {N}avier-{S}tokes and {D}arcy flows}, Computational
  Geosciences, 23 (2019), pp.~21--33.

\bibitem{layton2002coupling}
{\sc W.~J. Layton, F.~Schieweck, and I.~Yotov}, {\em Coupling fluid flow with
  porous media flow}, SIAM Journal on Numerical Analysis, 40 (2002),
  pp.~2195--2218.

\bibitem{liu2001energy}
{\sc J.-G. Liu and W.-C. Wang}, {\em An energy-preserving {MAC}--{Y}ee scheme
  for the incompressible {MHD} equation}, Journal of Computational Physics, 174
  (2001), pp.~12--37.

\bibitem{luo2017uzawa}
{\sc P.~Luo, C.~Rodrigo, F.~J. Gaspar, and C.~W. Oosterlee}, {\em Uzawa
  smoother in multigrid for the coupled porous medium and {S}tokes flow
  system}, SIAM Journal on Scientific Computing, 39 (2017), pp.~S633--S661.

\bibitem{mardal2022robust}
{\sc K.~A. Mardal, X.-C. Tai, and R.~Winther}, {\em A robust finite element
  method for {D}arcy--{S}tokes flow}, SIAM Journal on Numerical Analysis, 40
  (2002), pp.~1605--1631.

\bibitem{marquez2015strong}
{\sc A.~M{\'a}rquez, S.~Meddahi, and F.-J. Sayas}, {\em Strong coupling of
  finite element methods for the {S}tokes--{D}arcy problem}, IMA Journal of
  Numerical Analysis, 35 (2015), pp.~969--988.

\bibitem{mckee2004recent}
{\sc S.~Mckee, M.~F. Tom{\'e}, J.~A. Cuminato, A.~Castelo, and V.~G. Ferreira},
  {\em Recent advances in the marker and cell method}, Archives of
  computational methods in engineering, 11 (2004), pp.~107--142.

\bibitem{mckee2008mac}
{\sc S.~McKee, M.~F. Tom{\'e}, V.~G. Ferreira, J.~A. Cuminato, A.~Castelo,
  F.~Sousa, and N.~Mangiavacchi}, {\em The {MAC} method}, Computers $\&$
  Fluids, 37 (2008), pp.~907--930.

\bibitem{mu2007two}
{\sc M.~Mu and J.~Xu}, {\em A two-grid method of a mixed {S}tokes--{D}arcy
  model for coupling fluid flow with porous media flow}, SIAM Journal on
  Numerical Analysis, 45 (2007), pp.~1801--1813.

\bibitem{nicolaides1992analysis}
{\sc R.~A. Nicolaides}, {\em Analysis and convergence of the {MAC} scheme. {I}.
  the linear problem}, SIAM Journal on Numerical Analysis, 29 (1992),
  pp.~1579--1591.

\bibitem{pearson21}
{\sc J.~W. Pearson and A.~Potschka}, {\em On symmetric positive definite
  preconditioners for multiple saddle-point systems},  (2021).
\newblock https://arxiv.org/abs/2106.12433v3.

\bibitem{pearson21b}
\leavevmode\vrule height 2pt depth -1.6pt width 23pt, {\em A preconditioned
  inexact active-set method for large-scale nonlinear optimal control
  problems},  (2021).
\newblock https://arxiv.org/abs/2112.05020.

\bibitem{riviere2005locally}
{\sc B.~Rivi{\`e}re and I.~Yotov}, {\em Locally conservative coupling of
  {S}tokes and {D}arcy flows}, SIAM Journal on Numerical Analysis, 42 (2005),
  pp.~1959--1977.

\bibitem{rui2020mac}
{\sc H.~Rui and Y.~Sun}, {\em A {MAC} scheme for coupled {S}tokes--{D}arcy
  equations on non-uniform grids}, Journal of Scientific Computing, 82 (2020),
  pp.~1--29.

\bibitem{schmalfuss2021partitioned}
{\sc J.~Schmalfuss, C.~Riethm{\"u}ller, M.~Altenbernd, K.~Weishaupt, and
  D.~G{\"o}ddeke}, {\em Partitioned coupling vs. monolithic
  block-preconditioning approaches for solving {S}tokes-{D}arcy systems}, arXiv
  preprint arXiv:2108.13229,  (2021).

\bibitem{shiue2018convergence}
{\sc M.-C. Shiue, K.~C. Ong, and M.-C. Lai}, {\em Convergence of the {MAC}
  scheme for the {S}tokes/{D}arcy coupling problem}, Journal of Scientific
  Computing, 76 (2018), pp.~1216--1251.

\bibitem{sogn18}
{\sc J.~Sogn and W.~Zulehner}, {\em {Schur complement preconditioners for
  multiple saddle point problems of block tridiagonal form with application to
  optimization problems}}, IMA Journal of Numerical Analysis, 39 (2018),
  pp.~1328--1359.

\bibitem{sun2019stability}
{\sc Y.~Sun and H.~Rui}, {\em Stability and convergence of the mark and cell
  finite difference scheme for {D}arcy-{S}tokes-{B}rinkman equations on
  non-uniform grids}, Numerical Methods for Partial Differential Equations, 35
  (2019), pp.~509--527.

\bibitem{tlupova2022domain}
{\sc S.~Tlupova}, {\em A domain decomposition solution of the {S}tokes-{D}arcy
  system in {3D} based on boundary integrals}, Journal of Computational
  Physics, 450 (2022), p.~110824.

\bibitem{wang2019divergence}
{\sc G.~Wang, F.~Wang, L.~Chen, and Y.~He}, {\em A divergence free weak virtual
  element method for the {S}tokes--{D}arcy problem on general meshes}, Computer
  Methods in Applied Mechanics and Engineering, 344 (2019), pp.~998--1020.

\bibitem{zhang2009low}
{\sc S.~Zhang, X.~Xie, and Y.~Chen}, {\em Low order nonconforming rectangular
  finite element methods for {D}arcy-{S}tokes problems}, Journal of
  Computational Mathematics,  (2009), pp.~400--424.

\end{thebibliography}
\end{document}